\documentclass[12pt]{amsart}
\usepackage{cases}
\usepackage{mathrsfs}
\usepackage{color}
\usepackage{latexsym,amssymb}
\usepackage{a4wide}
\usepackage{amsthm}
\usepackage{amsmath}
\usepackage{graphicx}
\input xy
\xyoption{all}
\usepackage{verbatim}
\usepackage[lite,abbrev,msc-links]{amsrefs}
\usepackage{framed}
\usepackage{bm}

\numberwithin{equation}{section}
\begin{document}
	\theoremstyle{plain}
	\newtheorem{thm}{Theorem}[section]
	\newtheorem{lem}[thm]{Lemma}
	\newtheorem{cor}[thm]{Corollary}
	\newtheorem{cor*}[thm]{Corollary*}
	\newtheorem{prop}[thm]{Proposition}
	\newtheorem{prop*}[thm]{Proposition*}
	\newtheorem{conj}[thm]{Conjecture}
	\theoremstyle{definition}
	\newtheorem{construction}{Construction}
	\newtheorem{notations}[thm]{Notations}
	\newtheorem{question}[thm]{Question}
	\newtheorem{prob}[thm]{Problem}
	\newtheorem{rmk}[thm]{Remark}
	\newtheorem{remarks}[thm]{Remarks}
	\newtheorem{defn}[thm]{Definition}
	\newtheorem{claim}[thm]{Claim}
	\newtheorem{assumption}[thm]{Assumption}
	\newtheorem{assumptions}[thm]{Assumptions}
	\newtheorem{properties}[thm]{Properties}
	\newtheorem{exmp}[thm]{Example}
	\newtheorem{comments}[thm]{Comments}
	\newtheorem{blank}[thm]{}
	\newtheorem{observation}[thm]{Observation}
	\newtheorem{defn-thm}[thm]{Definition-Theorem}
	\newtheorem*{Setting}{Setting}

	\newcommand{\sA}{\mathscr{A}}
	\newcommand{\sB}{\mathscr{B}}
	\newcommand{\sC}{\mathscr{C}}
	\newcommand{\sD}{\mathscr{D}}
	\newcommand{\sE}{\mathscr{E}}
	\newcommand{\sF}{\mathscr{F}}
	\newcommand{\sG}{\mathscr{G}}
	\newcommand{\sH}{\mathscr{H}}
	\newcommand{\sI}{\mathscr{I}}
	\newcommand{\sJ}{\mathscr{J}}
	\newcommand{\sK}{\mathscr{K}}
	\newcommand{\sL}{\mathscr{L}}
	\newcommand{\sM}{\mathscr{M}}
	\newcommand{\sN}{\mathscr{N}}
	\newcommand{\sO}{\mathscr{O}}
	\newcommand{\sP}{\mathscr{P}}
	\newcommand{\sQ}{\mathscr{Q}}
	\newcommand{\sR}{\mathscr{R}}
	\newcommand{\sS}{\mathscr{S}}
	\newcommand{\sT}{\mathscr{T}}
	\newcommand{\sU}{\mathscr{U}}
	\newcommand{\sV}{\mathscr{V}}
	\newcommand{\sW}{\mathscr{W}}
	\newcommand{\sX}{\mathscr{X}}
	\newcommand{\sY}{\mathscr{Y}}
	\newcommand{\sZ}{\mathscr{Z}}
	\newcommand{\bZ}{\mathbb{Z}}
	\newcommand{\bN}{\mathbb{N}}
	\newcommand{\bQ}{\mathbb{Q}}
	\newcommand{\bC}{\mathbb{C}}
	\newcommand{\bR}{\mathbb{R}}
	\newcommand{\bH}{\mathbb{H}}
	\newcommand{\bD}{\mathbb{D}}
	\newcommand{\bE}{\mathbb{E}}
	\newcommand{\bP}{\mathbb{P}}
	\newcommand{\bV}{\mathbb{V}}
	\newcommand{\cV}{\mathcal{V}}
	\newcommand{\cF}{\mathcal{F}}
	\newcommand{\bfM}{\mathbf{M}}
	\newcommand{\bfN}{\mathbf{N}}
	\newcommand{\bfX}{\mathbf{X}}
	\newcommand{\bfY}{\mathbf{Y}}
	\newcommand{\spec}{\textrm{Spec}}
	\newcommand{\dbar}{\bar{\partial}}
	\newcommand{\ddbar}{\partial\bar{\partial}}
	\newcommand{\redref}{{\color{red}ref}}
	
	\title[Hyperbolicity of the base of an admissible family] {Hyperbolicity of the base of an admissible family of log canonical stable minimal models}
	
	\author[Junchao Shentu]{Junchao Shentu}
	\email{stjc@ustc.edu.cn}
	\address{School of Mathematical Sciences,
		University of Science and Technology of China, Hefei, 230026, China}
	\author[Chen Zhao]{Chen Zhao}
	\email{czhao@ustc.edu.cn}
	\address{School of Mathematical Sciences,
		University of Science and Technology of China, Hefei, 230026, China}

	\begin{abstract}
		We investigate the stratified hyperbolicity properties of Birkar's moduli stack of log canonical (lc) stable minimal models. The main technical result is a construction of Viehweg-Zuo's system of Higgs sheaves associated with an admissible family of lc stable minimal models, using the theory of degenerations of Hodge structure and non-abelian Hodge theory.
	\end{abstract}
	
	\maketitle
	
	\section{Introduction}
	The construction of compact moduli spaces of varieties is a fundamental problem in algebraic geometry. Since the seminal work of Deligne-Mumford \cite{Deligne1969} on the moduli of stable curves, a concerted effort has been made to generalize the construction of compact moduli spaces to higher dimensional varieties, with the aim of providing a framework for studying families of algebraic varieties with prescribed properties. These efforts have led to the development of a rich and diverse theory of moduli spaces, which has found applications in a wide range of areas in algebraic geometry and related fields. Known examples include the meaningful compactifications of the moduli spaces of polarized abelian varieties \cite{Alexeev2002}, plane curves \cite{Hacking2004},  manifolds of general type \cite{Kollar1988,Kollar2010,Kollar2018}, polarized Calabi-Yau manifolds \cite{KX2020} and $K$-polystable Fano manifolds \cite{Xu2020}. 
	Given $d\in\bN,c\in\bQ^{\geq0}$, a finite set $\Gamma\subset\bQ^{>0}$ and $\sigma\in\bQ[t]$, Birkar \cite{Birkar2022} introduced the notion of a $(d,\Phi_c,\Gamma,\sigma)$-stable minimal model as a triple $(X,B),A$, which consists of a variety $X$ over ${\rm Spec}(\bC)$ and $\bQ$-divisors $A\geq0$, $B\geq 0$ on $X$ satisfying the following conditions:
	\begin{itemize}
		\item $\dim X=d$, $(X,B)$ is an slc projective pair, $K_X+B$ is semi-ample,
		\item the coefficients of $A$ and $B$ are in $c\bZ^{\geq0}$,
		\item $(X,B+tA)$ is slc, $K_X+B+tA$ is ample for some $t>0$,
		\item ${\rm vol}(K_X+B+tA)=\sigma(t)$ for $0\leq t\ll 1$, and
		\item ${\rm vol}(A|_F)\in\Gamma$ where $F$ is any general fiber of the fibration $f:X\to Z$ determined by $K_X+B$.
	\end{itemize}
	The moduli stack $\sM_{\rm slc}(d,\Phi_c,\Gamma,\sigma)$ of $(d,\Phi_c,\Gamma,\sigma)$-stable minimal models (see \cite{Birkar2022} or \S \ref{section_moduli}) admits a projective good coarse moduli space (\cite[Theorem 1.14]{Birkar2022}) and provides a meaningful compactification of the moduli of birational equivalence classes of projective manifolds of an arbitrary Kodaira dimension. In the present article, we study the global geometry of the open substack $\sM_{{\rm lc},[0,1)}(d,\Phi_c,\Gamma,\sigma)$ of $\sM_{\rm slc}(d,\Phi_c,\Gamma,\sigma)$, which parameterizes the lc stable minimal models $(X,B),A$ where the coefficients of $B$ lie in $[0,1)$. These objects serve as canonical models of projective klt pairs.
	In recent years, the hyperbolicity properties of the moduli stack of varieties have attracted many attentions.
	The research in this area is pioneered by the works of Par\v{s}in \cite{Parsin1968},  Arakelov \cite{Arakelov1971} and a series of works of Viehweg-Zuo \cite{VZ2001,VZ2002,VZ2003}. In the present article we confirm the stratified hyperbolicity properties of $\sM_{{\rm lc},[0,1)}(d,\Phi_c,\Gamma,\sigma)$.
	
	A family of lc stable minimal models $(X^o,B^o),A^o\to S^o$ is called \emph{admissible} if the coefficients of $B^o$ lie in $[0,1)$ and $(X^o,A^o+B^o)$ admits a log smooth birational model (Definition \ref{defn_admissible}). 
	\begin{thm}\label{thm_main}
		Let $f^o:(X^o,B^o),A^o\to S^o$ be an admissible family of $(d,\Phi_c,\Gamma,\sigma)$-lc stable minimal models over a smooth quasi-projective variety $S^o$ which defines a morphism $\xi^o:S^o\to \sM_{\rm lc}(d,\Phi_c,\Gamma,\sigma)$.
		Then the following results hold. 
		\begin{itemize}
			\item\emph{(big Picard theorem)} Let $S$ be a projective variety containing $S^o$ as a Zariski open subset. Assume that $\xi^o$ is quasi-finite. Then $(S,S\backslash S^o)$ is a Picard pair\footnote{The classical big Picard theorem states that $(\bP^1,\{0,1,\infty\})$ is a Picard pair.}, in the sense that either $0\in\overline{\gamma^{-1}(S\backslash S^o)}$ or $\gamma$ can be extended to a holomorphic map $\overline{\gamma}:\Delta\to S$ for any given holomorphic map $\gamma:\Delta^\ast\to S$ from the punctured unit disc $\Delta^\ast$. Here $\Delta$ is the unit disk.
			\item\emph{(Borel hyperbolicity)} If $\xi^o$ is quasi-finite, then any holomorphic map from an algebraic variety to $S^o$ is algebraic. 
			\item\emph{(Viehweg hyperbolicity)} If $\xi^o$ is generically finite, then $S^o$ is of log general type in the sense that $\omega_S(S\backslash S^o)$ is big for any smooth compactification $S$ of $S^o$ such that $S\backslash S^o$ is a divisor. 
			\item\emph{(pseudo Kobayashi hyperbolicity)} If $\xi^o$ is generically finite, then $S^o$ is Kobayashi hyperbolic away from a proper Zariski closed subset $Z\subset S^o$, meaning that the Kobayashi pseudo-distance $d_{S^o}$ satisfies $d_{S^o}(x,y)>0$ for any distinct $x,y\in S^o$ that are not both in $Z$.
			\item\emph{(Brody hyperbolicity)} If $\xi^o$ is quasi-finite, then $S^o$ is Brody hyperbolic, that is, there is no non-constant holomorphic map $\bC\to S^o$.
		\end{itemize}
	\end{thm}
	The following examples show that the assumption 'admissible' in Theorem \ref{thm_main} is indispensable.
	\begin{enumerate}
		\item \emph{The presence of the degenerating fiber.} Let $f:X\to\bP^1$ be a Lefschetz pencil with $S\subset\bP^1$ the set of its critical values, such that the general fibers of $f$ are canonically polarized $d$-folds with $v$ their volumes. Then $f:(X,0),0\to\bP^1$ is a family of $(d,\Phi_0,\{1\},v)$-lc stable minimal models.  $\bP^1$ is not hyperbolic while $\bP^1\backslash S$ is hyperbolic ($\#(S)\geq3$ due to \cite{VZ2001}).
		\item \emph{The presence of the degenerating polarization.} Let $E$ be an elliptic curve and  $x_0\in E(\bC)$. Let $X=E\times E$ and let $f:E\times E\to E$ denote the projection to the first component. Let  $A=\frac{1}{2}(E\times\{x_0\}\cup\Delta_E)$ where $\Delta_E\subset E\times E$ is the diagonal. Then $f:(X,0),A\to E$ is a family of $(1,\Phi_{\frac{1}{2}},\{1\},t)$-lc stable minimal models. The underlying family of elliptic curves is trivial, however the family of polarizations is non-isotrivial, degenerating at $x_0$. The entire base $E$ is not hyperbolic, whereas $E\backslash\{x_0\}$ is hyperbolic.
	\end{enumerate}
	A direct consequence of Theorem \ref{thm_main} is that any finite covering projective variety $S\to\sM_{{\rm lc},[0,1)}(d,\Phi_c,\Gamma,\sigma)$ is stratified hyperbolic with respect to the admissible condition (Theorem \ref{thm_stratified_hyperbolic}). 
	\begin{exmp}[Log smooth families of projective pairs of log general type]
		Let $f:(X,B)\to S$ be a log smooth family of projective klt pairs of log general type. Assume that $\dim X_s=d$, ${\rm vol}(K_{X_s}+B_s)=v$ for each $s\in S$, and the coefficients of $B$ lie in $c\bZ^{\geq0}$ for some $c\in \bQ^{\geq0}$. Then the relative lc model $(X^{\rm can},B^{\rm can}),0\to S$ (c.f. \cite{BCHM2010}) is an admissible family of $(d,\Phi_c,\{1\},v)$-lc stable minimal models (see \cite[Page 721]{WeiWu2023}). Furthermore, $f$ induces a morphism $S\to\sM_{\rm lc}(d,\Phi_c,\{1\},v)$.  Theorem \ref{thm_main} yields the following corollary.
		\begin{cor}
			Fix $d\in\bN,c\in\bQ^{\geq0}$ and $v\in\bQ^{>0}$.
			Let $f^o:(X^o,B^o)\to S^o$ be a log smooth family of projective pairs of general type over a smooth quasi-projective variety $S^o$ such that $\dim X_s=d$, ${\rm vol}(K_{X_s}+B_s)=v$ for every $s\in S^o$ and the coefficients of $B$ lie in $c\bZ^{\geq0}\cap[0,1)$. Let $\xi^o:S^o\to \sM_{\rm lc}(d,\Phi_c,\{1\},v)$ be the map determined by $f^o$.
			Then the following results hold. 
			\begin{itemize}
				\item Let $S$ be a projective variety containing $S^o$ as a dense Zariski open subset. If $\xi^o$ is quasi-finite, then $(S,S\backslash S^o)$ is a Picard pair.
				\item If $\xi^o$ is quasi-finite, then $S^o$ is Borel hyperbolic and Brody hyperbolic.
				\item If $\xi^o$ is generically finite, then $S^o$ is of log general type and Kobayashi hyperbolic modulo a proper Zariski closed subset.
			\end{itemize}    
		\end{cor}
		Among the results above, the Viehweg hyperbolicity, the Brody hyperbolicity and the pseudo Kobayashi hyperbolicity have been proved, mainly due to \cite{VZ2003,PS17,PTW2019,WeiWu2023}.
	\end{exmp}
	\begin{exmp}[Families of stable Calabi-Yau pairs]
		A lc stable minimal model $(X,B),A$ is a stable Calabi-Yau pair if $K_X+B\sim_{\bQ}0$. The base of an admissible family of $(d,\Phi_c,\Gamma,\sigma)$-lc stable Calabi-Yau pairs satisfies the hyperbolicity properties proposed in Theorem \ref{thm_main}. 
	\end{exmp}
	\begin{exmp}[Families of stable Fano pairs]
		A lc stable minimal model $(X,B),A$ is a stable Fano pair if $(X,A+B),A$ is a stable Calabi-Yau pair. The base of an admissible family of $(d,\Phi_c,\Gamma,\sigma)$-stable Fano pairs satisfies the hyperbolicity properties proposed in Theorem \ref{thm_main}. 
	\end{exmp}
	Many efforts have been made to the hyperbolicity properties of various  moduli spaces of polarized projective manifolds. For an incomplete list see \cite{VZ2003,PS17,PTW2019,Deng2022,DLSZ,WeiWu2023} and the references therein. 
	
	The \emph{Viehweg hyperbolicity} in Theorem \ref{thm_main} has been studied by numerous scholars, with a non-exhaustive list of works including but not limited to \cite{Kovacs1996,Kovacs2000,OV2001,VZ2001,Kovacs2002,Kovacs2003,KK2008,KK20082,KK2010,Patakfalvi2012,PS17,CP2019,Taji2021,WeiWu2023}.
	The Viehweg hyperbolicity, in the case of families of canonically polarized manifolds (known as the Viehweg hyperbolicity conjecture), was first proved by Campana-P\v{a}un \cite{CP2019}, based on the so-called
	Viehweg-Zuo sheaves constructed in \cite{VZ2003}. By using Saito's theory of Hodge modules \cite{MSaito1988}, Popa-Schnell \cite{PS17} extended the construction of Viehweg-Zuo sheaves to base spaces of an arbitrary family of projective varieties with a geometric generic fiber that admits a good minimal model, and proved the relevant hyperbolicity property for these families. Popa-Schnell's result was further generalized by Wei-Wu \cite{WeiWu2023} to log smooth families of pairs of log general type. There are also other admissible conditions to ensure the relevant hyperbolicity results, see \cite{Kovacs2021,Park2022} for example.
	
	The \emph{big Picard theorem} can be traced back to the classical big Picard theorem, which states that any holomorphic map from the punctured disk  $\Delta^\ast$ into $\bP^1$ that omits three points can be extended to a holomorphic map $\Delta\rightarrow \bP^1$. In a recent work by Deng-Lu-Sun-Zuo \cite{DLSZ}, it has been confirmed that the big Picard theorem holds for the base of a maximal variational family of good minimal manifolds. 
	The \emph{Borel hyperbolicity} is a formal consequence of the big Picard theorem (see \cite[Corollary C]{DLSZ} or Proposition \ref{prop_BPT_BH}).
	
	The \emph{pseudo Kobayashi hyperbolicity}, in the case of family of curves of genus $g>1$, was established by Royden \cite{Royden1975} and Wolpert \cite{Wolpert1986}. To-Yeung \cite{Yeung2015} made the first breakthrough on higher dimensional families, demonstrating that the base manifold of any effectively parametrized family of canonically polarized manifolds is Kobayashi hyperbolic. Further relevant works on other families can be found in Berndtsson-P\u{a}un-Wang \cite{BPW2022}, Schumacher \cite{Schumacher2018}, Deng \cite{Deng2022} and To-Yeung \cite{Yeung2018,Yeung2021}.

	The \emph{Brody hyperbolicity}, in the case of families of canonically polarized manifolds, was
	proved by Viehweg-Zuo \cite{VZ2003} using their construction of systems of Higgs sheaves. It was generalized by Popa-Taji-Wu \cite{PTW2019} to families of general type minimal manifolds, using similar constructions as in \cite{PS17}. These results were further generalized by Deng \cite{Deng2022} to families of good minimal manifolds, answering a question by Viehweg-Zuo \cite[Question 0.1]{VZ2003}.
	Wei-Wu \cite{WeiWu2023} studied the Brody hyperbolicity of a log smooth family of pairs of log general type. 
	
	Due to the works \cite{VZ2003,CP2019,Deng2022,DLSZ}, the five hyperbolicity results follow from the construction of a certain system of Higgs sheaves associated with the relevant family. The main contribution of the present article is the construction of the Viehweg-Zuo type system of Higgs sheaves associated with an admissible family of lc stable minimal models.
	\begin{thm}[=Theorem \ref{thm_big_Higgs_sheaf}]\label{thm_main_VZHiggs}
		Let $f^o:(X^o,B^o),A^o\to S^o$ be an  admissible family of $(d,\Phi_c,\Gamma,\sigma)$-lc stable minimal models over a smooth quasi-projective variety $S^o$ which defines a generically finite morphism $\xi^o:S^o\to M_{\rm lc}(d,\Phi_c,\Gamma,\sigma)$.
		Let $S$ be a smooth projective variety containing $S^o$ as a Zariski open subset such that $D:=S\backslash S^o$ is a (reduced) simple normal crossing divisor and $\xi^o$ extends to a morphism $\xi:S\to M_{\rm lc}(d,\Phi_c,\Gamma,\sigma)$\footnote{We do not require $\xi$ to have a moduli interpretation at the boundary $D$.}. Let $\sL$ be a line bundle on $S$. Then there exist the following data.
		\begin{enumerate}
			\item A projective birational morphism $\pi:S'\to S$ such that $S'$ is smooth, $\pi^{-1}(D)$ is a simple normal crossing divisor and $\pi$ is a composition of smooth blow-ups.
			\item A (possibly non-reduced) effective exceptional divisor $E$ of $\pi$ such that $E\cup\pi^{-1}(D)$ has a simple normal crossing support.
			\item A $\bQ$-polarized variation of Hodge structure of weight $w>0$ on $S'\backslash (E\cup \pi^{-1}(D))$ with $(H=\bigoplus_{p+q=w} H^{p,q},\theta,h)$ its associated Higgs bundle by taking the total graded quotients with respect to the Hodge filtration. Here $h$ is the Hodge metric.
		\end{enumerate}
		These data satisfy the following conditions.
		\begin{enumerate}
			\item There is a coherent ideal sheaf $I_Z$ on $S$ whose co-support $Z$ is contained in $D$ and ${\rm codim}_S(Z)\geq2$, and a natural inclusion $\sL\otimes I_Z\subset \pi_{\ast}\left({_{<\pi^{-1}(D)+E}}H^{w,0}\right)$.
			\item Let $(\bigoplus_{p=0}^w L^p,\theta)$ be the log Higgs subsheaf generated by $L^0:=\sL\otimes I_Z$, where $$L^p\subset\pi_{\ast}\left({_{<\pi^{-1}(D)+E}}H^{w-p,p}\right).$$ Then  the Higgs field $$\theta:L^p|_{S\backslash (D\cup\pi(E))}\to L^{p+1}|_{S\backslash (D\cup\pi(E))}\otimes \Omega_{S\backslash (D\cup\pi(E))}$$ is holomorphic over $S\backslash D$ for each $0\leq p<n$, that is, 
			$$\theta(L^p)\subset L^{p+1}\otimes\Omega_S(\log D).$$
		\end{enumerate}
	\end{thm}
	Here ${_{<\pi^{-1}(D)+E}}H^{p,q}$ is the prolongation of the Hodge bundle $H^{p,q}$ in the sense of Simpson \cite{Simpson1990} and Mochizuki \cite{Mochizuki20072} (see \S \ref{section_prolongation}), consisting of the meromorphic sections with prescribed poles. 
	Theorem \ref{thm_main_VZHiggs} is proved by extending the original construction of Viehweg-Zuo \cite{VZ2003} by using the theory of degenerations of Hodge structure and non-abelian Hodge theory. Our construction is also valid for K\"ahler morphisms (see \S\ref{section_Analytic_prolongation_VZ}).
	
	The present article is organized as follows. In \S2 we recall the theory of degenerations and prolongations of a variation of Hodge structure. In \S3 we investigate the analytic prolongations of Viehweg-Zuo Higgs sheaves (Theorem \ref{thm_VZ_prolongation}). \S4 is the review of the basic notions and constructions of Birkar's moduli space of stable minimal models. Theorem \ref{thm_main_VZHiggs} is proved in \S \ref{section_VZHiggs_stable_family}. \S5 is devoted to the proof of Theorem \ref{thm_main} and several consequences.
	
	{\bf Notations:}
	\begin{itemize}
		\item Let $F$ be a torsion free coherent sheaf on a smooth variety $X$. Let  $F^{\vee}:=\sH om_{\sO_X}(F,\sO_X)$ so that $F^{\vee\vee}$ is the reflexive hull of $F$. We use $\det(F)$ to denote the reflexive hull of $\wedge^{{\rm rank}(F)}F$.
		\item 
		We define a functorial desingularization of $(X,Z)$, where $X$ is a complex space and $Z\subset X$ is a closed analytic subset, in the sense of W{\l}odarczyk \cite{Wlodarczyk2005,Wlodarczyk2009}. Specifically, this desingularization is a projective bimeromorphic morphism $\pi:X'\to X$ from a complex manifold $X'$, satisfying the conditions that $\pi^{-1}(Z)$, $\pi^{-1}(X_{\rm sing})$, and $\pi^{-1}(Z)\cup\pi^{-1}(X_{\rm sing})$ are simple normal crossing divisors on $X'$, and $\pi$ is biholomorphic over the loci where $(X,Z)$ is a log smooth pair. 
		\item  $\Delta:=\{t\in\bC\mid |t|<1\}$ denotes the unit disck in $\bC$ and $\Delta^\ast:=\Delta\setminus\{0\}$.
		\item Let $f:X\to S$ be a morphism between algebraic varieties. The fiber $f^{-1}\{s\}$ over the point $s\in S$ is denoted by $X_s$.
	\end{itemize}
	\section{Analytic prolongations of variations of Hodge structure}
	\subsection{Norm estimates for the Hodge metric}
	Let $\bV=(\cV,\nabla,\cF^\bullet,Q)$ be an $\bR$-polarized variation of Hodge structure over $(\Delta^\ast)^n\times \Delta^m$, together with the Hodge metric $h_Q$, where $(\cV,\nabla)$ is a flat connection, $\cF^\bullet$ is the Hodge filtration and $Q$ is a real polarization. Let us fix the standard coordinates $s_1,\dots,s_n, w_1,\dots,w_m$ on $(\Delta^\ast)^n\times \Delta^m$. Let $D_i:=\{s_i=0\}\subset\Delta^{n+m}$ for every $i=1,\dots,n$. Let $N_i$ be the unipotent part of ${\rm Res}_{D_i}\nabla$ and let 
	$$p:\bH^{n}\times \Delta^m\to (\Delta^\ast)^n\times \Delta^m,$$ 
	$$(z_1,\dots,z_n,w_1,\dots,w_m)\mapsto(e^{2\pi\sqrt{-1}z_1},\dots,e^{2\pi\sqrt{-1}z_n},w_1,\dots,w_m)$$
	be the universal covering. Let
	$W^{(1)}=W(N_1),\dots,W^{(n)}=W(N_1+\cdots+N_n)$ be the monodromy weight filtrations (centered at 0) on $V:=\Gamma(\bH^n\times \Delta^m,p^\ast\cV)^{p^\ast\nabla}$.
	The following important norm estimates for flat sections are due to Cattani-Kaplan-Schmid \cite[Theorem 5.21]{Cattani_Kaplan_Schmid1986} and Mochizuki \cite[Part 3, Chapter 13]{Mochizuki20072}.
	\begin{thm}\label{thm_Hodge_metric_asymptotic}
		For any $0\neq v\in {\rm Gr}_{l_n}^{W^{(n)}}\cdots{\rm Gr}_{l_1}^{W^{(1)}}V$, one has
		\begin{align*}
		|v|^2_{h_Q}\sim \left(\frac{\log|s_1|}{\log|s_2|}\right)^{l_1}\cdots\left(-\log|s_n|\right)^{l_n}
		\end{align*}
		on any region of the form
		$$\left\{(s_1,\dots, s_n,w_1,\dots,w_m)\in (\Delta^\ast)^n\times \Delta^m\bigg|\frac{\log|s_1|}{\log|s_2|}>\epsilon,\dots,-\log|s_n|>\epsilon,(w_1,\dots,w_m)\in K\right\}$$
		for any $\epsilon>0$ and  $K\subset \Delta^m$ compact.
	\end{thm}
	The rest of this part is devoted to the norm estimates of $h_Q$ on $S(\bV)$, where $S(\bV)$ denotes $\cF^{\max\{p\mid\cF^p\neq0\}}$. Let $\cV_{-1}$ denote Deligne's canonical extension of $(\cV,\nabla)$ whose real parts of the eigenvalues of the residue maps lie in $(-1,0]$. By the nilpotent orbit theorem \cite{Cattani_Kaplan_Schmid1986} $j_\ast S(\bV)\cap\cV_{-1}$ is a subbundle of $\cV_{-1}$.
	\begin{lem}\label{lem_W_F}
		Assume that $n=1$. Then $W_{-1}(N_1)\cap \big(j_\ast S(\bV)\cap\cV_{-1}\big)|_{\bf 0}=0$.
	\end{lem}
	\begin{proof}
		Assume that $W_{-1}(N_1)\cap \big(j_\ast S(\bV)\cap\cV_{-1}\big)|_{\bf 0}\neq0$ and let $k$ be the weight of $\bV$. Let $l=\max\{l\mid W_{-l}(N_1)\cap \big(j_\ast S(\bV)\cap\cV_{-1}\big)|_{\bf 0}\neq0\}$. Then $l\geq 1$. 
		By \cite[6.16]{Schmid1973}, the filtration $j_\ast \cF^\bullet\cap\cV_{-1}$ induces a pure Hodge structure of weight $m+k$ on $W_{m}(N_1)/W_{m-1}(N_1)$. Moreover, 
		\begin{align}\label{align_hard_lef_N}
		N^l: W_{l}(N_1)/W_{l-1}(N_1)\to W_{-l}(N_1)/W_{-l-1}(N_1)
		\end{align}
		is an isomorphism of type $(-l,-l)$. Let $p=\max\{i\mid \cF^i\neq0\}$. By the definition of $l$, any nonzero element $\alpha\in W_{-l}(N_1)\cap \big(j_\ast S(\bV)\cap\cV_{-1}\big)|_{\bf 0}$ induces a nonzero $[\alpha]\in W_{-l}(N_1)/W_{-l-1}(N_1)$ of Hodge type $(p,k-l-p)$. Since (\ref{align_hard_lef_N}) is an isomorphism, there is $\beta\in W_{l}(N_1)/W_{l-1}(N_1)$ of Hodge type $(p+l,k-p)$ such that $N^l(\beta)=[\alpha]$. However, $\beta=0$ since $\cF^{p+l}=0$. This contradicts to the fact that $[\alpha]\neq0$. Consequently, $W_{-1}(N_1)\cap \big(j_\ast S(\bV)\cap\cV_{-1}\big)_{\bf 0}$ must be zero.
	\end{proof}
	Let $T_i$ denote the local monodromy operator of $\bV$ around $D_i$.
	Since $T_1,\dots,T_n$ are pairwise commutative, there is a finite decomposition 
	$$\cV_{-1}|_{\bf 0}=\bigoplus_{-1<\alpha_1,\dots,\alpha_n\leq 0}\bV_{\alpha_1,\dots,\alpha_n}$$
	such that $(T_i-e^{2\pi\sqrt{-1}\alpha_i}{\rm Id})$ is unipotent on $\bV_{\alpha_1,\dots,\alpha_n}$ for each $i=1,\dots,n$. 
	Let $$v_1,\dots, v_N\in (\cV_{-1}\cap j_\ast S(\bV))|_{\bf 0}\cap\bigcup_{-1<\alpha_1,\dots,\alpha_n\leq 0}\bV_{\alpha_1,\dots,\alpha_n}$$
	be an orthogonal basis of $(\cV_{-1}\cap j_\ast S(\bV))|_{\bf 0}\simeq \Gamma(\bH^n\times\Delta^m,p^\ast S(\bV))^{p^\ast\nabla}$. Then $\widetilde{v_1},\dots,\widetilde{v_N}$ that are determined by
	\begin{align}\label{align_adapted_frame}
	\widetilde{v_j}:={\rm exp}\left(\sum_{i=1}^n\log s_i(\alpha_i{\rm Id}+N_i)\right)v_j\textrm{ if } v_j\in\bV_{\alpha_1,\dots, \alpha_n},\quad \forall j=1,\dots,N
	\end{align}
	form a frame of $\cV_{-1}\cap j_\ast S(\bV)$.
	We always use the notation $\alpha_{D_i}(\widetilde{v_j})$ instead of $\alpha_i$ in (\ref{align_adapted_frame}). By (\ref{align_adapted_frame}) we see that 
	\begin{align*}
	|\widetilde{v_j}|^2_{h_Q}&\sim\left|\prod_{i=1}^ns_i^{\alpha_{D_i}(\widetilde{v_j})}{\rm exp}\left(\sum_{i=1}^nN_i\log s_i\right)v_j\right|^2_{h_Q}\\\nonumber
	&\sim|v_j|^2_{h_Q}\prod_{i=1}^n |s_i|^{2\alpha_{D_i}(\widetilde{v_j})},\quad j=1,\dots,N
	\end{align*}
	where $\alpha_{D_i}(\widetilde{v_j})\in(-1,0]$, $\forall i=1,\dots, n$. 
	It follows from Theorem \ref{thm_Hodge_metric_asymptotic} and Lemma \ref{lem_W_F}
	that
	\begin{align*}
	|v_j|^2_{h_Q}\sim \left(\frac{\log|s_1|}{\log|s_2|}\right)^{l_1}\cdots\left(-\log|s_n|\right)^{l_n}\quad\textrm{with}\quad 0\leq l_1\leq l_2\leq\dots\leq l_{n},
	\end{align*}
	on any region of the form
	$$\left\{(s_1,\dots, s_n,w_1,\dots,w_{m})\in (\Delta^\ast)^n\times \Delta^{m}\bigg|\frac{\log|s_1|}{\log|s_2|}>\epsilon,\dots,-\log|s_n|>\epsilon,(w_1,\dots,w_{m})\in K\right\}$$
	for any $\epsilon>0$ and $K\subset \Delta^{m}$ compact. Therefore, we know that
	\begin{align*}
	1\lesssim |v_j|\lesssim|s_1\cdots s_n|^{-\epsilon},\quad\forall\epsilon>0.
	\end{align*}
	
	\begin{defn}(Zucker \cite[page 433]{Zucker1979})
		Let $(E,h)$ be a vector bundle with a possibly singular hermitian metric $h$ on a hermitian manifold $(X,ds^2_0)$. A holomorphic local frame $(v_1,\dots,v_N)$ of $E$ is called $L^2$-adapted if, for every set of measurable functions $\{f_1,\dots,f_N\}$, $\sum_{i=1}^Nf_iv_i$ is locally square integrable if and only if $f_iv_i$ is locally square integrable for all $i$.
	\end{defn}
	\begin{lem}
		The local frame $(\widetilde{v_1},\dots,\widetilde{v_N})$ is $L^2$-adapted.
	\end{lem}
	\begin{proof}
		If 
		$$\sum_{j=1}^N f_j\widetilde{v_j}={\rm exp}\left(\sum_{i=1}^nN_i\log s_i\right)\left(\sum_{j=1}^N f_j\prod_{i=1}^n |s_i|^{\alpha_{D_i}(\widetilde{v_j})}v_j\right)$$
		is locally square integrable, then 
		$$\sum_{j=1}^N f_j\prod_{i=1}^n |s_i|^{\alpha_{D_i}(\widetilde{v_j})}v_j$$
		is locally square integrable because the entries of the matrix ${\rm exp}\left(-\sum_{i=1}^nN_i\log s_i\right)$ are $L^\infty$-bounded.
		Since $(v_1,\dots,v_N)$ is an orthogonal basis, 
		$|f_j\widetilde{v_j}|_{h_Q}\sim\prod_{i=1}^n |s_i|^{\alpha_{D_i}(\widetilde{v_j})}|f_jv_j|_{h_Q}$ is locally square integrable for all $j$. 
	\end{proof}
	In conclusion, we obtain the following proposition.
	\begin{prop}\label{prop_adapted_frame}
		Let $(X,ds^2_0)$ be a hermitian manifold and $D$ a normal crossing divisor on $X$. Let $\bV$ be an $\bR$-polarized variation of Hodge structure on $X^o:=X\backslash D$. Then there is an $L^2$-adapted holomorphic local frame $(\widetilde{v_1},\dots,\widetilde{v_N})$ of $\cV_{-1}\cap j_\ast S(\bV)$ at every point $x\in D$. 
		Let $z_1,\cdots,z_n$ be holomorphic local coordinates on $X$ so that $D
		=\{z_1\cdots z_r=0\}$. Then there are constants  $\alpha_{D_i}(\widetilde{v_j})\in(-1,0]$, $i=1,\dots, r$, $j=1,\dots,N$ and positive real functions $\lambda_j\in C^\infty(X\backslash D)$, $j=1,\dots,N$ such that
		\begin{align*}
		|\widetilde{v_j}|^2\sim\lambda_j\prod_{i=1}^r |z_i|^{2\alpha_{D_i}(\widetilde{v_j})},\quad \forall j=1,\dots,N
		\end{align*}
		and
		$$1\lesssim \lambda_j\lesssim|z_1\cdots z_r|^{-\epsilon},\quad\forall\epsilon>0,\quad \forall j=1,\dots,N$$
	\end{prop}
	\subsection{Prolongations of a variation of Hodge structure: log smooth case}\label{section_prolongation}
	Let $X$ be a complex manifold and $D=\sum_{i=1}^l D_i$ a reduced simple normal crossing divisor on $X$.
	Let $(E,h)$ be a holomorphic vector bundle on $X\backslash D$ with a smooth hermitian metric $h$. Let $D_1=\sum_{i=1}^la_iD_i$, $D_2=\sum_{i=1}^lb_iD_i$ be $\bR$-divisors. We denote $D_1<(\leq) D_2$ if $a_i<(\leq) b_i$ for all $i$.
	\begin{defn}[Analytic prolongation](\cite{Mochizuki2002}, Definition 4.2)\label{defn_prolongation}
		Let $A=\sum_{i=1}^la_iD_i$ be an $\bR$-divisor, let $U$ be an open subset of $X$, and let $s\in \Gamma(U\backslash D,E)$ be a holomorphic section. We denote $(s)\leq -A$ if $|s|_h=O(\prod_{k=1}^r |z_{k}|^{-a_{i_k}-\epsilon})$ for any positive number $\epsilon$, where $z_1,\dots,z_n$ are holomorphic local coordinates such that $D=\{z_1\cdots z_r=0\}$ and $D_{i_k}=\{z_k=0\}$ for each $k=1,\dots,r$.
		The $\sO_X$-module $ _{A}E$ is defined as 
		$$\Gamma(U, {_{A}}E):=\{s\in\Gamma(U\backslash D,E)|(s)\leq -A\}$$
		for any open subset $U\subset X$.
		Let
		\begin{align*}
		{_{<A}}E:=\bigcup_{{B}<{A}}{_{B}}E\quad\textrm{and}\quad{\rm Gr}_{A}E:={_{A}}E/{_{<A}}E.
		\end{align*}	
	\end{defn}
	Let $\bV=(\cV,\nabla,\cF^\bullet,Q)$ be an $\bR$-polarized variation of Hodge structure of weight $w$ on $X\backslash D$. Let $(H:={\rm Gr}_{\cF^\bullet}\cV,\theta:={\rm Gr}_{\cF^\bullet}\nabla)$ denote the total graded quotient. Then $(H,\theta)$ is the Higgs bundle  corresponding to $(\cV,\nabla)$ via Simpson's correspondence \cite{Simpson1988}. The Hodge metric $h_Q$ associated with $Q$ is a harmonic metric on $(H,\theta)$. The triple $(H,\theta,h_Q)$ is a tame harmonic bundle in the sense of Simpson \cite{Simpson1990} and Mochizuki \cite{Mochizuki20072}. Notice that $(H,\theta)$ is a system of Hodge bundles (\cite[\S 4]{Simpson1992}) in the sense that
	\begin{align*}
	H=\bigoplus_{p+q=w} H^{p,q},\quad H^{p,q}\simeq \cF^{p}/\cF^{p+1},\quad \theta(H^{p,q})\subset H^{p-1,q+1}\otimes\Omega_{X\backslash D}.
	\end{align*}
	According to Simpson \cite[Theorem 3]{Simpson1990} and Mochizuki \cite[Proposition 2.53]{Mochizuki2009},  the set of prolongations forms a parabolic structure.
	\begin{thm}\label{thm_parabolic}
		Let $X$ be a complex manifold and $D=\sum_{i=1}^l D_i\subset X$ a reduced simple normal crossing divisor. Let $(H=\oplus_{p+q=w} H^{p,q},\theta,h_Q)$ be the system of Hodge bundles associated with an $\bR$-polarized variation of Hodge structure of weight $w$ on $X\backslash D$.
		For each $\bR$-divisor $A$ supported on $D$, $_{A}H$ is a locally free coherent sheaf satisfying the following conditions:
		\begin{itemize}
			\item  $_{A+\epsilon D_i}H = {_A}H$ for any $i=1,\dots,l$ and any constant $0<\epsilon\ll 1$,
			\item $_{A+D_i}H={_A}H\otimes \sO(-D_i)$ for every $1\leq i\leq l$,
			\item the subset of $(a_1,\dots,a_l)\in\bR^l$ such that ${\rm Gr}_{\sum_{i=1}^l a_iD_i}H\neq 0$ is discrete, and
			\item the Higgs field $\theta$ has at most logarithmic poles along $D$, that is, $\theta$ extends to 
			\begin{align*}
			{_{A}}H\to{_{A}}H\otimes\Omega_{X}(\log D).
			\end{align*}
		\end{itemize} 
	\end{thm}
	The proof of the following lemma is straightforward and thus omitted here.
	\begin{lem}\label{lem_integral}
		Let $f$ be a holomorphic function on $\Delta^\ast:=\{z\in\bC|0<|z|<1\}$ and $a\in\bR$. Then
		$$\int_{|z|\leq\frac{1}{2}}|f|^2|z|^{2a}dzd\bar{z}<\infty$$
		if and only if $v(f)+a>-1$. Here
		$$v(f):=\min\{l|f_l\neq0\textrm{ in the Laurent expansion } f=\sum_{i\in\bZ}f_iz^i\}.$$
	\end{lem}
	\begin{lem}\label{lem_Deligne_prolongation}
		Notations as above. Let $S(\bV)=\cF^{\max\{p|\cF^p\neq0\}}$. Then there is a natural isomorphism $$\cV_{-1}\cap j_\ast S(\bV)\simeq {_{<D}}S(\bV),$$
		where $j:X\backslash D\to X$ is the immersion and ${_{<D}}S(\bV)$ is taken with respect to the Hodge metric $h_Q$. Let $U\subset X$ be an open subset. Then a holomorphic section $s\in S(\bV)(U\backslash D)$ extends to a section in ${_{<D}}S(\bV)(U)$ if and only if the integration
		$$\int|s|^2_{h_Q}{\rm vol}_{ds^2}$$
		is finite locally at every point of $U\cap D$, where $ds^2$ is a hermitian metric on $X$.
	\end{lem}
	\begin{proof}
		It follows from Proposition \ref{prop_adapted_frame} and Lemma \ref{lem_integral} that 
		$$\cV_{-1}\cap j_\ast S(\bV)\subset {_{<D}}S(\bV).$$
		For the converse, let $\widetilde{v_1},\dots,\widetilde{v_N}$ be the $L^2$-adapted local frame of $\cV_{-1}\cap j_\ast S(\bV)$ (as in Proposition \ref{prop_adapted_frame}) at some point $x\in D$. Let $\alpha=\sum_{j=1}^N f_j\widetilde{v_j}\in {_{<D}}S(\bV)$ where $f_1,\dots,f_N$ are functions that are holomorphic outside $D$. By Lemma \ref{lem_integral}, $\alpha$ is locally square integrable at $x$. Hence, all $f_j\widetilde{v_j}$ are locally square integrable at $x$. According to \ref{prop_adapted_frame} and Lemma \ref{lem_integral}, it follows that the functions $f_1,\dots,f_N$ are holomorphic in some neighborhood of $x$. This proves
		$${_{<D}}S(\bV)\subset \cV_{-1}\cap j_\ast S(\bV)$$
		and the last claim of the lemma.
	\end{proof}
	\subsection{Prolongations of a variation of Hodge structure: general case}\label{section_prolongation_general} 
	The analytic prolongation of a variation of Hodge structure on a general base is defined via desingularization. 
	Let $X$ be a complex manifold and $Z\subset X$ a closed analytic subset. Let $D\subset Z$ be the union of the irreducible components of $Z$ whose codimension is one. Let $\pi:\widetilde{X}\to X$ be a functorial desingularization of the pair $(X,Z)$ so that $\widetilde{X}$ is smooth, $\pi^{-1}(Z)$ is a simple normal crossing divisor on $\widetilde{X}$ and 
	$$\pi^o:=\pi|_{\widetilde{X}^o}:\widetilde{X}^o:=\pi^{-1}(X\backslash Z)\to X^o:=X\backslash Z$$
	is biholomorphic.
	Let $\bV=(\cV,\nabla,\cF^\bullet,Q)$ be an $\bR$-polarized variation of Hodge structure of weight $w$ on $\widetilde{X}^o$ and $(H=\oplus_{p+q=w}H^{p,q},\theta,h_Q)$ the corresponding Higgs bundle with the Hodge metric $h_Q$. Let $A$ be an $\bR$-divisor supported on $\pi^{-1}(Z)$. Then $\pi_\ast({_{A}}H)$ is a torsion free coherent sheaf on $X$ whose restriction on $X^o$ is $(\pi^o)^{-1\ast}(H)$. By abuse of notation we still denote $\theta:=(\pi^o)^{-1\ast}(\theta)$. $\theta$ is a meromorphic Higgs field on $\pi_\ast({_{A}}H)$ with poles along $Z$. Let ${\rm Cryt}(\pi)\subset X$ be the degenerate loci of $\pi$. Since $\pi$ is functorial, $D\backslash {\rm Cryt}(\pi)$ is a simple normal crossing divisor on $X\backslash {\rm Cryt}(\pi)$ and the exceptional loci $\pi^{-1}({\rm Cryt}(\pi))$ is a simple normal crossing divisor on $\widetilde{X}$. $(\pi_\ast({_{A}}H),\theta)|_{X\backslash {\rm Cryt}(\pi)}$ is locally free and $\theta$ admits at most log poles along $D\backslash {\rm Cryt}(\pi)$. The following negativity result for $\ker(\theta)$ generalizes \cite{Zuo2000}. The main idea of its proof is due to Brunebarbe \cite{Brunebarbe2017}.
	\begin{prop}\label{prop_semipositive_kernel}
		Notations as above. Assume that ${\rm supp}(A)$ lies in the exceptional divisor $\pi^{-1}({\rm Cryt}(\pi))$. Let $K\subset \pi_\ast({_{A}}H)$ be a coherent subsheaf such that $\theta(K)=0$. Then $K^\vee$ is weakly positive in the sense of Viehweg \cite{Viehweg1983}.
	\end{prop}
	\begin{proof}
		We follow the notion of singular hermitian metrics on torsion free coherent sheaves (as in \cite{Paun2018,HPS2018}).
		The Hodge metric $h_Q$ defines a singular hermitian metric on the bundle $_{A}H$, with singularities along $\pi^{-1}(Z)$. Since $\pi^o$ is biholomorphic, we may regard $h_Q$ as a singular hermitian metric on the torsion free coherent sheaf $\pi_\ast({_{A}}H)$. Let $K^o:=K|_{X^o}$. By Griffiths' curvature formula
		$$\Theta_{h_Q}(H)+\theta\wedge\overline{\theta}+\overline{\theta}\wedge\theta=0,$$
		one knows that
		$$\Theta_{h_Q}(K^o)=-\theta\wedge\overline{\theta}|_{K^o}+\overline{B}\wedge B$$
		is Griffiths semi-negative, where $B\in A^{1,0}_{X^o}(K^o,K^{o\bot})$ is the second fundamental class. We claim that the hermitian metric $h_Q|_{X^o}$ extends to a singular hermitian metric on $K$ with semi-negative curvatures. It suffices to prove that $\log |s|_{h_Q}$ can be extended to a plurisubharmonic function on $X$ for an arbitrary section $s\in K$.
		
		Since $\Theta_{h_Q}(K^o)$ is Griffiths semi-negative, $\log|s|_{h_Q}$ is a smooth plurisubharmonic function on $X^o$. By  Riemannian extension theorem and Hartogs extension theorem for plurisubharmonic functions \cite[Lemma 12.4]{HPS2018}, it suffices to show that $\log|s|_{h_Q}$ is locally bounded from above in codimension one. Let ${\rm Cryt}(\pi)\subset X$ be the degenerate loci of $\pi$, which is of codimension $\geq 2$. Then $D\backslash {\rm Cryt}(\pi)$ is a simple normal crossing divisor on $X\backslash {\rm Cryt}(\pi)$. The assumption on $A$ yields that $\pi_\ast({_{A}}H)|_{X\backslash {\rm Cryt}(\pi)}\simeq{_{\bm{0}}}((\pi^{o})^{-1\ast}H)$, where $\bm{0}$ is the zero divisor on $X\backslash {\rm Cryt}(\pi)$. Let $x$ be a general point of a component $D_i$ of $D\backslash {\rm Cryt}(\pi)$. Let $N_i$ be the monodromy operator around $D_i$ associated with the connection $((\pi^{o})^{-1\ast}\cV,(\pi^{o})^{-1\ast}\nabla)$ and let $\{W_k\}_{k\in\bZ}$ be the monodromy weight filtration determined by $N_i$. Since $\theta(s)=0$, one has $s\in W_0$ according to \cite[Corollary 6.7]{Schmid1973} (see also \cite[Lemma 5.4]{Brunebarbe2017}). Combining it with the fact that $s\in {_{\bm{0}}}((\pi^{o})^{-1\ast}H)$, it follows from Simpson's norm estimate \cite[page 721]{Simpson1990} that $|s|_{h_Q}$ is locally bounded near $x$. This implies the claim that $h_Q$ extends (uniquely) to a singular hermitian metric on $K$ with semi-negative curvature. Hence $K^{\vee}$ is weakly positive in the sense of Viehweg by \cite[Theorem 2.5.2]{PT2018}.
	\end{proof}
	\subsection{Prolongations of a variation of Hodge structure of geometric origin}
	Let $f:Y\to X$ be a proper holomorphic morphism between complex manifolds and let $n:=\dim X-\dim Y$. Let $Z\subset X$ be a closed analytic subset such that $f$ is a K\"ahler submersion over $X^o:=X\backslash Z$. Let $Y^o:=f^{-1}(X^o)$ and $f^o:=f|_{Y^o}:Y^o\to X^o$. Then $R^nf^o_\ast(\bR_{Y^o})$ underlies an $\bR$-polarized variation of Hodge structure $\bV^n_{f^o}=(\cV^n,\nabla,\cF^\bullet,Q)$ of weight $n$. Here $\cV^n\simeq R^nf^o_\ast(\bR_{Y^o})\otimes_{\bR}\sO_{X^o}$, $\nabla$ is the Gauss-Manin connection, $\cF^p\simeq R^nf^o_\ast(\Omega^{\geq p}_{Y^o/X^o})$ and $Q$ is the $\bR$-polarization associated with a relative K\"ahler form. Let $h_Q$ be the Hodge metric associated with $Q$ and let $(H^n_{f^o}=\oplus_{p+q=n}H^{p,q}_{f^o},\theta)$ be the Higgs bundle associated with $\bV^n_{f^o}$ where $H^{p,q}_{f^o}\simeq R^qf^o_\ast(\Omega^p_{Y^o/X^o})$. Let $\omega_{Y/X}:=\omega_Y\otimes f^{\ast}(\omega_X^{-1})$ be the relative dualizing sheaf.
	\begin{lem}\label{lem_prolongation0_vs_geo}
		Notations as above. If $Z$ is a reduced simple normal crossing divisor, then there is an isomorphism $$f_\ast(\omega_{Y/X})\simeq {_{<Z}}H^{n,0}_{f^o}.$$
	\end{lem}
	\begin{proof}
		Let $j:X^o\to X$ be the open immersion.
		It suffices to show that
		${_{<Z}}H^{n,0}_{f^o}\otimes\omega_X=f_\ast(\omega_{Y})$ as subsheaves of $j_\ast H^{n,0}_{f^o}\otimes\omega_X$. Let $s$ be a local section of $ j_\ast H^{n,0}_{f^o}=j_\ast f^o_\ast(\omega_{Y^o/X^o})$. Let $\phi=dz_1\wedge\cdots\wedge dz_{d}$ where $z_1,\dots,z_d$ are holomorphic local coordinates on $X$. According to Lemma \ref{lem_Deligne_prolongation}, $s\in {_{<Z}}H^{n,0}_{f^o}$ if and only if the integral
		$$\int_{X^o}|s|^2_{h_Q}\phi\wedge\overline{\phi}=\epsilon_n\int_{X^o}\left(\int_{f^{-1}\{x\}}s|_{f^{-1}\{x\}}\wedge\overline{s|_{f^{-1}\{x\}}}\right)\phi\wedge\overline{\phi}=\epsilon_n\int_{Y^o}(s\wedge f^{o\ast}(\phi))\wedge\overline{s\wedge f^{o\ast}(\phi)}$$
		is finite locally at every point of $Z$, where $\epsilon_n=(-1)^{\frac{n(n-1)}{2}}(\sqrt{-1})^n$.
		The locally finiteness of the right handside is equivalent to that $s\wedge f^{o\ast}(\phi)$ admits a holomorphic extension to $Y$ (c.f. \cite[Proposition 16]{Kawamata1981}). This proves that  ${_{<Z}}H^{n,0}_{f^o}\otimes\omega_X=f_\ast(\omega_{Y})$.
	\end{proof}
	Let us return to the general case. Consider the diagram
	\begin{align}
		\xymatrix{
			Y'\ar[r]^{\sigma} \ar[d]^{f'} & Y \ar[d]^f\\
			X'\ar[r]^{\pi} &X
		}
	\end{align}
	such that the following conditions hold.
	\begin{itemize}
		\item $\pi:X'\to X$ is a desingularization of the pair $(X,Z)$. In particular, $X'$ is smooth, $\pi^{-1}(Z)$ is a simple normal crossing divisor and $\pi^o:=\pi|_{\pi^{-1}(X^o)}:\pi^{-1}(X^o)\to X^o$ is biholomorphic.
		\item $Y'$ is a functorial desingularization of the main component of $Y\times_XX'$. In particular $Y'\to Y\times_XX'$ is biholomorphic over $f^{-1}(X^o)\times_{X^o}\pi^{-1}(X^o)$.
	\end{itemize}
	Let $\omega_{X'}\simeq\pi^{\ast}\omega_X\otimes\sO_{X'}(E)$
	for some exceptional divisor $E$ of $\pi$. We obtain the natural morphisms
	\begin{align}\label{align_pullback_pushforward0}
		\pi^\ast(f_\ast(\omega_{Y/X}))\simeq\pi^\ast(f_\ast(\omega_Y)\otimes\omega_X^{-1})\to f'_\ast(\omega_{Y'})\otimes\omega^{-1}_{X'}\otimes\sO_{X'}(E)\simeq f'_\ast(\omega_{Y'/X'})\otimes\sO_{X'}(E).
	\end{align}
	Define $$f'^o:=f'|_{\sigma^{-1}(f^{-1}(X^o))}:\sigma^{-1}(f^{-1}(X^o))\to\pi^{-1}(X^o),$$  a proper K\"ahler submersion since $\pi^o$ is biholomorphic. Let  $H^n_{f'^o}$ be the Higgs bundle associated with $f'^o$. Lemma \ref{lem_prolongation0_vs_geo} yields that
	\begin{align*}
		f'_\ast(\omega_{Y'/X'})\simeq {_{<\pi^{-1}(Z)_{\rm red}}}H^{n,0}_{f'^o}.
	\end{align*}
	Combining it with (\ref{align_pullback_pushforward0}), we obtain a generically injective morphism
	$$f_\ast(\omega_{Y/X})\to \pi_\ast({_{<\pi^{-1}(Z)_{\rm red}}}H^{n,0}_{f'^o}\otimes \sO_{X'}(E))\simeq \pi_\ast({_{<\pi^{-1}(Z)_{\rm red}+E}}H^{n,0}_{f'^o}).$$
	Since $f_\ast(\omega_{Y/X})$ is torsion free, the above map must be injective. Thus we have concluded the following result.
	\begin{prop}\label{prop_prolongation0_vs_geo}
		Notations as above. Then there is an inclusion $$f_\ast(\omega_{Y/X})\subset \pi_\ast({_{<\pi^{-1}(Z)_{\rm red}+E}}H^{n,0}_{f'^o}).$$
	\end{prop}
	\section{Analytic prolongations of Viehweg-Zuo Higgs sheaves}\label{section_Analytic_prolongation_VZ}
	In this section we generalize Viehweg-Zuo's construction of Higgs sheaves using analytic prolongations (Theorem \ref{thm_VZ_prolongation}). 
	\subsection{Setting}\label{section_setting}
	Throughout this section let us fix a proper holomorphic morphism $f:Y\to X$ between complex manifolds with  $n=\dim Y-\dim X$ the relative dimension. Assume that there is a simple normal crossing divisor $D_f\subset X$ such that $f^o:=f|_{Y^o}:Y^o\to X^o$ is a K\"ahler submersion where $X^o:=X\backslash D_f$ and $Y^o:=f^{-1}(X^o)$. We fix a torsion free coherent sheaf $L$ on $X$ that is invertible on $X^o$ (hence ${\rm rank}(L)=1$), and a nonzero morphism
	\begin{align}\label{align_s}
		s_L:L^{\otimes k}\to f_\ast(\omega_{Y/X}^{\otimes k})
	\end{align}
	for some $k\geq 1$. 
	\subsection{Viehweg-Zuo Higgs sheaves}\label{section_VZ_sheaf}
	Notations as in \S \ref{section_setting}. Let $L^{\vee\vee}$ be the reflexive hull of $L$, with $L\to L^{\vee\vee}$ the natural inclusion map. Since ${\rm rank}(L^{\vee\vee})=1$, $L^{\vee\vee}$ is an invertible sheaf. Since $L$ is torsion free and is invertible on $X^o$, $\sI_T:=L\otimes(L^{\vee\vee})^{-1}\subset\sO_X$ is a coherent ideal sheaf whose co-support lies in a closed analytic subset $T\subset D_f$ such that ${\rm codim}_X(T)\geq 2$. Consider a diagram
	\begin{align}\label{align_setting}
		\xymatrix{
			\widetilde{Y}\ar[r]^{\sigma}\ar[d]_{\tilde{f}} & Y\ar[d]^f\\
			\widetilde{X}\ar[r]^{\pi} & X
		}
	\end{align}
	of holomorphic maps between complex manifolds such that the following conditions hold. 
	\begin{itemize}
		\item $\pi$ is a functorial desingularization of $(X,T,D_f)$ in the sense of W{\l}odarczyk \cite{Wlodarczyk2009}. In particular,  $\widetilde{X}$ is a compact complex manifold, $\pi$ is a projective morphism that is biholomorphic over $X\backslash T$. $\pi^{-1}(D_f)$, $E:=\pi^{-1}(T)_{\rm red}$ and $\pi^{-1}(D_f)\cup E$ are simple normal crossing divisors. 
		\item $\widetilde{Y}$ is a functorial desingularization of the main component of $Y\times_X\widetilde{X}$. In particular,  $\widetilde{Y}\to Y\times_X\widetilde{X}$ is biholomorphic over $f^{-1}(X\backslash T)\times_{X\backslash T}\pi^{-1}(X\backslash T)$.
	\end{itemize}
	Since $\pi$ is biholomorphic on $\widetilde{X}\backslash E$, there is a constant $k_0\geq 0$ and a natural map
	\begin{align*}
		\pi^{\ast}f_\ast(\omega_{Y/X}^{\otimes k})\otimes\sO_{\widetilde{X}}(-k_0kE)\to \widetilde{f}_\ast(\omega^{\otimes k}_{\widetilde{Y}/\widetilde{X}}).
	\end{align*}
	Taking (\ref{align_s}) into account, we obtain a non-zero morphism
	\begin{align*}
		\pi^{\ast}(L^{\vee\vee})^{\otimes k}\otimes\pi^\ast(\sI_T)^{\otimes k}\simeq\pi^{\ast}L^{\otimes k}\to \pi^\ast(f_\ast(\omega_{Y/X}^{\otimes k}))\to \widetilde{f}_\ast(\omega^{\otimes k}_{\widetilde{Y}/\widetilde{X}})\otimes\sO_{\widetilde{X}}(k_0kE). 
	\end{align*}
	Hence there is an effective divisor $\widetilde{E}$, supported on $E$, such that there is a nonzero map
	\begin{align*}
		\pi^{\ast}(L^{\vee\vee})^{\otimes k}\otimes\sO_{\widetilde{X}}(-k\widetilde{E})\to  \widetilde{f}_\ast(\omega^{\otimes k}_{\widetilde{Y}/\widetilde{X}}). 
	\end{align*}
	Let $\widetilde{L}:=\pi^{\ast}(L^{\vee\vee})\otimes\sO_{\widetilde{X}}(-\widetilde{E})$ and $L^o:=L|_{X^o}$. Let $\pi^o:=\pi|_{\pi^{-1}(X^o)}:\pi^{-1}(X^o)\to X^o$. The arguments above show that there is a non-zero morphism
	\begin{align}
		s_{\widetilde{L}}:\widetilde{L}^{\otimes k}\to \widetilde{f}_\ast(\omega^{\otimes k}_{\widetilde{Y}/\widetilde{X}})
	\end{align} 
	and an isomorphism
	\begin{align}\label{align_iso_Lo_L}
		\pi^{o\ast}(L^o)\simeq\widetilde{L}|_{\pi^{-1}(X^o)}
	\end{align} 
	such that the diagram 
	\begin{align}\label{align_commut_s_s0}
		\xymatrix{
			\widetilde{L}^{\otimes k}|_{\pi^{-1}(X^o)}\ar[rr]^-{s_{\widetilde{L}}|_{\pi^{-1}(X^o)}}&& \widetilde{f}_\ast(\omega^{\otimes k}_{\widetilde{Y}/\widetilde{X}})|_{\pi^{-1}(X^o)}
			\\
			\pi^{o\ast}(L^o)^{\otimes k}\ar[u]^{\simeq}\ar[rr]^-{\pi^{o\ast}(s_L|_{X^o})} && \pi^{o\ast}f^o_\ast(\omega^{\otimes k}_{Y^o/X^o})\ar[u]^{\simeq}
		}
	\end{align}
	is commutative. 
	Define 
	$$B^o=\omega_{Y^o/X^o}\otimes f^{o\ast}(L^{o})^{-1},$$
	a line bundle on $Y^o$ and 
	$$\widetilde{B}=\omega_{\widetilde{Y}/\widetilde{X}}\otimes \tilde{f}^\ast(\widetilde{L}^{-1}),$$
	a line bundle on $\widetilde{Y}$. Then the map $s_{\widetilde{L}}$ determines a non-zero section
	$\widetilde{s}\in H^0(\widetilde{Y},\widetilde{B}^{\otimes k})$.
	Let $\varpi:\widetilde{Y}_k\to \widetilde{Y}$ be the $k:1$ cyclic covering map that is branched along $\{\widetilde{s}=0\}$ and let $\mu:Z\to \widetilde{Y}_k$ be a functorial desingularization that is biholomorphic over the complement of $\{\varpi^\ast\widetilde{s}=0\}$.  The morphisms are gathered in the following diagram
	\begin{align*}
		\xymatrix{
			Z\ar[r]^{\mu}\ar[drr]^g & \widetilde{Y}_k \ar[r]^{\varpi} & \widetilde{Y}\ar[d]^{\widetilde{f}}\ar[r]^{\sigma}& Y\ar[d]^f\\
			&& \widetilde{X}\ar[r]^\pi & X
		},
	\end{align*}
	where $g$ denotes $\widetilde{f}\varpi\mu$.
	Let $D_g\subset\widetilde{X}$ be a reduced closed analytic subset  containing $\pi^{-1}(D_f)$, such that $g$ is a submersion over $\widetilde{X}^o:=\widetilde{X}\backslash D_g$. Let $Z^o:=g^{-1}(\widetilde{X}^o)$ and let $g^o:=g|_{Z^o}:Z^o\to \widetilde{X}^o$. Since $\mu$ and $\varpi$ are projective morphisms, $g^o$ is a proper K\"ahler submersion. 
	Consider the diagram
	\begin{align}
		\xymatrix{
			Z'\ar[r]_{\sigma'} \ar[d]^{h} \ar@/^/[rr]^{\varphi}& Z \ar[d]^g\ar[r]_{\sigma\varpi\mu}& Y\ar[d]^f\\
			X'\ar[r]^{\rho} \ar@/_/[rr]_{\psi}&\widetilde{X}\ar[r]^\pi &X
		}
	\end{align}
	where $\varphi:=\sigma\varpi\mu\sigma'$ and $\psi:=\pi\rho$, such that the following conditions hold.
	\begin{itemize}
		\item $\rho:X'\to \widetilde{X}$ is a functorial desingularization of $(\widetilde{X},D_g,\pi^{-1}(D_f))$. In particular, $X'$ is smooth, $\rho^{-1}(D_g)$, $\psi^{-1}(D_f)$ and $\rho^{-1}(D_g)\cup \psi^{-1}(D_f)$ are simple normal crossing divisors and $\rho^o:=\rho|_{\rho^{-1}(\widetilde{X}^o)}:\rho^{-1}(\widetilde{X}^o)\to \widetilde{X}^o$ is biholomorphic.
		\item $Z'$ is a functorial desingularization of the main component of $Z\times_{\widetilde{X}}X'$. In particular $Z'\to Z\times_{\widetilde{X}}X'$ is biholomorphic over $Z^o\times_{\widetilde{X}^o}\rho^{-1}(\widetilde{X}^o)$.
	\end{itemize}
	Let $X'^o:=\rho^{-1}(\widetilde{X}^o)$, $Z'^o:=h^{-1}(X'^o)$ and $h^o:=h|_{Z'^o}:Z'^o\to X'^o$.
	Notice that $h^o$ is a proper K\"ahler submersion of relative dimension $n$, which is the pullback of the family $g^o:Z^o\to\widetilde{X}^o$ via the isomorphism $\rho^o:X'^o\to\widetilde{X}^o$.
	Then $R^nh^o_\ast(\bR_{Z'^o})$ underlies an $\bR$-polarized variation of Hodge structure of weight $n$ on $X'^o$. Let $(H^n_{h^o}=\bigoplus_{p=0}^n H^{p,n-p}_{h^o},\theta,h_Q)$ be the associated system of Hodge bundles with the Hodge metric $h_Q$. Namely, $H^{p,q}_{h^o}:= R^qh^o_\ast\Omega^p_{Z'^o/X'^o}$ and
	$\theta:H^{p,q}_{h^o}\to H^{p-1,q+1}_{h^o}\otimes\Omega_{X'^o}$ is defined by taking wedge product with the Kodaira-Spencer class. 
	
	Let $\omega_{X'}\simeq\rho^{\ast}\omega_{\widetilde{X}}\otimes\sO_{X'}(E')$
	for some exceptional divisor $E'$ of $\rho$.
	By Theorem \ref{thm_parabolic} and \S \ref{section_prolongation_general},
	$$\left(\psi_{\ast}\left({_{<\psi^{-1}(D_f)_{\rm red}+E'}}H^n_{h^o}\right)=\bigoplus_{p=0}^n \psi_{\ast}\left({_{<\psi^{-1}(D_f)_{\rm red}+E'}}H^{p,n-p}_{h^o}\right),\theta\right)$$is a meromorphic Higgs sheaf on $X$ such that the Higgs field $\theta$ is holomorphic over $X\backslash\pi(D_g)$ and is regular along $\pi(D_g)$.
	
	The main result of this subsection is the following theorem, inspired by the constructions in \cite{Zuo2001}.
	\begin{thm}\label{thm_VZ_prolongation}
		Notations and assumptions as in \S \ref{section_setting} and \S \ref{section_VZ_sheaf}. Then the following hold.
		\begin{enumerate}
			\item There is a natural inclusion $\pi_\ast(\widetilde{L})\subset \psi_{\ast}\left({_{<\psi^{-1}(D_f)_{\rm red}+E'}}H^{n,0}_{h^o}\right)$.
			\item Let $(\bigoplus_{p=0}^n L^p,\theta)\subset \left(\psi_{\ast}\left({_{<\psi^{-1}(D_f)_{\rm red}+E'}}H^n_{h^o}\right),\theta\right)$ be the meromorphic Higgs subsheaf generated by $L^0:=\pi_\ast(\widetilde{L})$, where $$L^p\subset\psi_{\ast}\left({_{<\psi^{-1}(D_f)_{\rm red}+E'}}H^{n-p,p}_{h^o}\right).$$ Then the Higgs field $$\theta:L^p|_{X\backslash \pi(D_g)}\to L^{p+1}|_{X\backslash \pi(D_g)}\otimes \Omega_{X\backslash \pi(D_g)}$$ is holomorphic over $X\backslash D_f$ and has at most log poles along $D_f$ for each $0\leq p<n$, that is, 
			$$\theta(L^p)\subset L^{p+1}\otimes\Omega_X(\log D_f).$$
		\end{enumerate}
	\end{thm}
	The proof will occupy the remainder of this subsection. It will be accomplished by constructing a log Higgs subsheaf $\bigoplus_{p+q=n}G^{p,q}$ of $\psi_{\ast}\left({_{<\psi^{-1}(D_f)_{\rm red}+E'}}H^n_{h^o}\right)$ which contains $\pi_\ast(\widetilde{L})$ such that the Higgs field is holomorphic on $X\backslash D_f$. We first construct the Higgs subsheaf on $\psi(X'^o)$ and then extend it to the whole manifold $X$ by using analytic prolongations.
	\subsubsection{The construction on $\psi(X'^o)$}
	Let $X_1:=\psi(X'^o)\subset X^o$ and  $Y_1:=f^{-1}(X_1)\subset Y^o$. Then $f_1:=f|_{Y_1}:Y_1\to X_1$ is a proper K\"ahler submersion. Let $\varphi^o:=\varphi|_{Z'^o}:Z'^o\to Y_1$.
	Since $\widetilde{Y}_k$ is embedded into the total space of the line bundle $\widetilde{B}$, the pullback $(\varpi\mu)^\ast \widetilde{B}$ has a tautological section. This gives an injective morphism
	\begin{align*}
		(\varpi\mu\sigma')^\ast (\widetilde{B}^{-1})\to\sO_{Z'}.
	\end{align*}
	Combining it with (\ref{align_iso_Lo_L}), one gets an injective map 
	$$\varphi^{o\ast}(B^o|_{Y_1})^{-1}\simeq(\varpi\mu\sigma')^\ast (\widetilde{B}^{-1})|_{Z'^o}\to\sO_{Z'^o}.$$
	By composing it with the natural map $\varphi^{o\ast}\Omega^p_{Y_1/X_1}\to\Omega^p_{Z'^o/X'^o}$, we obtain a natural morphism
	\begin{align}\label{align_VZ1}
		\varphi^{o\ast}((B^o)^{-1}\otimes\Omega^p_{Y^o/X^o}|_{Y_1})\to \Omega^p_{Z'^o/X'^o}
	\end{align}
	for every $p=0,\dots,n$. 
	Hence (\ref{align_VZ1}) induces a map
	\begin{align}\label{align_VZmap}
		\iota_{X_1}:R^qf^o_{\ast}((B^o)^{-1}\otimes\Omega^{p}_{Y^o/X^o})|_{X_1}\to \psi^o_\ast R^qh^o_\ast(\Omega^{p}_{Z'^o/X'^o})
	\end{align} 
	for every $p,q\geq0$, where $\psi^o:=\psi|_{X'^o}:X'^o\to X_1$ is an isomorphism.   
	Consider the diagram
	\begin{align*}
		\xymatrix{
			0\ar[r]&	h^{o\ast}\Omega_{X'^o}\otimes\Omega^{p-1}_{Z'^o/X'^o}\ar[r]&\Omega^{p}_{Z'^o} \ar[r] &\Omega^{p}_{Z'^o/X'^o}\ar[r]&0
		}.
	\end{align*}
	By taking the higher direct image $R^\ast h^o_\ast$, we obtain the Higgs field as the coboundary map
	\begin{align}\label{align_theta_log}
		\theta:R^qh^o_\ast(\Omega^{p}_{Z'^o/X'^o})\to R^{q+1}h^o_\ast(\Omega^{p-1}_{Z'^o/X'^o})\otimes\Omega_{X'^o}.
	\end{align}
	Consider the diagram
	\begin{align*}
		\xymatrix{
			0\ar[r]&	f^{o\ast}\Omega_{X^o}\otimes\Omega^{p-1}_{Y^o/X^o}\ar[r]&\Omega^{p}_{Y^o} \ar[r] &\Omega^{p}_{Y^o/X^o}\ar[r]&0
		}.
	\end{align*}
	By tensoring it with $(B^o)^{-1}$ and taking the higher direct image $R^\ast f^o_{\ast}$, one has the coboundary map
	\begin{align}\label{align_vartheta_nolog}
		\vartheta:R^qf^o_{\ast}((B^o)^{-1}\otimes\Omega^{p}_{Y^o/X^o})\to R^{q+1}f^o_{\ast}((B^o)^{-1}\otimes\Omega^{p-1}_{Y^o/X^o})\otimes\Omega_{X^o}.
	\end{align}
	It follows from (\ref{align_VZ1}) that there is a morphism between distinguished triangles in the derived category $D(Y_1)$
	\begin{align*}
		\xymatrix{
			f^{o\ast}\Omega_{X^o}\otimes\Omega^{p-1}_{Y^o/X^o}\otimes (B^o)^{-1}|_{Y_1}\ar[r]\ar[d]&\Omega^{p}_{Y^o}\otimes (B^o)^{-1}|_{Y_1} \ar[r]\ar[d] &\Omega^{p}_{Y^o/X^o}\otimes (B^o)^{-1}|_{Y_1}\ar[r]\ar[d]&\\
			R\varphi^o_\ast\left(h^{o\ast}\Omega_{X'^o}\otimes\Omega^{p-1}_{Z'^o/X'^o}\right)\ar[r]&R\varphi^o_\ast\left(\Omega^{p}_{Z'^o}\right) \ar[r] &R\varphi^o_\ast\left(\Omega^{p}_{Z'^o/X'^o}\right)\ar[r]&
		}.
	\end{align*}
	Then there is a commutative diagram
	\begin{align}\label{align_theta_vartheta_commute}
		\xymatrix{
			R^qf^o_{\ast}((B^o)^{-1}\otimes\Omega^{p}_{Y^o/X^o})|_{X_1}\ar[r]^-{\vartheta|_{X_1}}\ar[d]^{\iota_{X_1}}& R^{q+1}f^o_{\ast}((B^o)^{-1}\otimes\Omega^{p-1}_{Y^o/X^o})|_{X_1}\otimes\Omega_{X_1}\ar[d]^{\iota_{X_1}\otimes{\rm Id}}\\
			\psi^o_\ast R^qh^o_\ast(\Omega^{p}_{Z'^o/X'^o})\ar[r]^-{\theta}& \psi^o_\ast R^{q+1}h^o_\ast(\Omega^{p-1}_{Z'^o/X'^o})\otimes\Omega_{X_1}
		}.
	\end{align}
	\subsubsection{Extend the Higgs sheaves to $X^o$}
	Notice that $H^{p,q}_{h^o}\simeq R^qh^o_\ast(\Omega^p_{Z'^o/X'^o})$. The main result of this part is the following lemma.
	\begin{lem}\label{lem_VZ_Xo}
		The map (\ref{align_VZmap}) extends to a map
		\begin{align}\label{align_VZmap2}
			\iota_{X^o}:R^qf^o_{\ast}((B^o)^{-1}\otimes\Omega^{p}_{Y^o/X^o})\to \psi_\ast\left( _{<\rho^{-1}(D_g)_{\rm red}}H^{p,q}_{h^o}\right)|_{X^o}.
		\end{align}
	\end{lem}
	\begin{proof}
		Consider the diagram
		\begin{align}\label{align_ssr_lem3_3}
			\xymatrix{
				Z''\ar[r]^-{\beta}&\varphi^{-1}(Y^o)\ar[r]^-\varphi \ar[d]^{h'} & Y^o\ar[d]^{f^o}\\
				&\psi^{-1}(X^o)\ar[r]^-\psi & X^o
			}.
		\end{align}
		where $h':=h|_{\varphi^{-1}(Y^o)}$ is a proper K\"ahler submersion over $X'^o=\psi^{-1}(X^o)\backslash\rho^{-1}(D_g)$ and
		$\beta:Z''\to \varphi^{-1}(Y^o)$ is a functorial desingularization of the pair $(\varphi^{-1}(Y^o),h'^{-1}(\psi^{-1}(X^o)\backslash X'^o))$. Notice that there is a closed analytic subset $S\subset\psi^{-1}(X^o)\backslash X'^o$ so that ${\rm codim}_{\psi^{-1}(X^o)}(S)\geq 2$ and $h'\beta:Z''\to \psi^{-1}(X^o)$ is semistable (\S \ref{section_ssred}) over $\psi^{-1}(X^o)\backslash S$.
		(\ref{align_ssr_lem3_3}) induces the natural morphisms
		\begin{align}\label{align_lem33_1}
			\psi^\ast\left(R^qf^o_\ast((B^o)^{-1}\otimes\Omega^p_{Y^o/X^o})\right)\to R^qh'_\ast(\Omega^p_{\varphi^{-1}(Y^o)/\psi^{-1}(X^o)})\to R^q(h'\beta)_\ast\left(\Omega^p_{Z''/\psi^{-1}(X^o)}\right)
		\end{align}
		for every $p,q\geq0$. Let
		$X'_2:=\psi^{-1}(X^o)\backslash S$, $D_{X'_2}:=\rho^{-1}(D_g)\cap X'_2$, $Z''_2:=(h'\beta)^{-1}(X'_2)$ and $D_{Z''_2}:=(h'\beta)^{-1}(D_{X_2})_{\rm red}$. Then $h'\beta|_{Z''_2}:(Z''_2,D_{Z''_2})\to(X'_2,D_{X'_2})$ is a proper K\"ahler semistable morphism (\S \ref{section_ssred}).
		Consider the associated logarithmic Gauss-Manin connection 
		\begin{align*}
			\nabla_{\rm GM}:R^m(h'\beta|_{Z''_2})_\ast\left(\Omega^\bullet_{Z''_2/X'_2}(\log D_{Z''_2})\right)\to R^m(h'\beta|_{Z''_2})_\ast\left(\Omega^\bullet_{Z''_2/X'_2}(\log D_{Z''_2})\right)\otimes\Omega_{X'_2}(\log D_{X'_2})
		\end{align*}
	    where $0\leq m\leq 2n$.
		According to \cite[Proposition 2.2]{Steenbrink1975}, the real parts of the eigenvalues of the residue map of $\nabla_{\rm GM}$ along each component of $D_{X'_2}$ lie in $[0,1)$. As a consequence, the corresponding logarithmic Higgs bundle lies in the prolongation $_{<D_{X'_2}}H^m$,
		where
		$$H^m:=\bigoplus_{p+q=m}R^q(h'\beta|_{Z''_2\backslash D_{Z''_2}})_\ast\left(\Omega^p_{(Z''_2\backslash D_{Z''_2})/X'^o}\right)$$
		is the Higgs bundle associated with the proper K\"ahler submersion $Z''_2\backslash D_{Z''_2}\to X'^o$. Namely there is a natural inclusion
		\begin{align*}
			R^q(h'\beta|_{Z''_2})_\ast\left(\Omega^p_{Z''_2/X'_2}(\log D_{Z''_2})\right)\to _{<D_{X'_2}}R^q(h'\beta|_{Z''_2\backslash D_{Z''_2}})_\ast\left(\Omega^p_{(Z''_2\backslash D_{Z''_2})/X'^o}\right), \quad\forall p,q\geq0.
		\end{align*}
		Since $\beta$ is an isomorphism over the open submanifold $h'^{-1}(X'^o)$, the family $Z''_2\backslash D_{Z''_2}\to X'^o$ is isomorphic to the family $h^o:Z'^o\to X'^o$. Consequently, one obtains a natural inclusion
		\begin{align*}
			R^q(h'\beta|_{Z''_2})_\ast\left(\Omega^p_{Z''_2/X'_2}(\log D_{Z''_2})\right)\to _{<\rho^{-1}(D_g)_{\rm red}}H^{p,q}_{h^o}|_{X'_2}, \quad\forall p,q\geq0.
		\end{align*}
		Taking (\ref{align_lem33_1}) into account, one gets a map
		\begin{align}\label{align_pq_to_prolongation}
			\psi^\ast\left(R^qf^o_\ast((B^o)^{-1}\otimes\Omega^p_{Y^o/X^o})\right)\big|_{X'_2}\to _{<\rho^{-1}(D_g)_{\rm red}}H^{p,q}_{h^o}|_{X'_2}, \quad\forall p,q\geq0.
		\end{align}
		Since $_{<\rho^{-1}(D_g)_{\rm red}}H^{p,q}_{h^o}$ is locally free (Theorem \ref{thm_parabolic}) and ${\rm codim}_{\psi^{-1}(X^o)}(S)\geq 2$, the morphism (\ref{align_pq_to_prolongation}) extends to a morphism
		\begin{align*}
			\psi^\ast\left(R^qf^o_\ast((B^o)^{-1}\otimes\Omega^p_{Y^o/X^o})\right)\to _{<\rho^{-1}(D_g)_{\rm red}}H^{p,q}_{h^o}|_{\psi^{-1}(X^o)}, \quad p,q\geq0.
		\end{align*}
		by Hartogs extension theorem. Taking the adjoint we obtain (\ref{align_VZmap2}).
	\end{proof}
	\subsubsection{Extend the Higgs sheaves to $X$}
	In this part we extend (\ref{align_VZmap2}) to $X$. Let
	$$R^qf^o_{\ast}((B^o)^{-1}\otimes\Omega^{p}_{Y^o/X^o})\langle D_f\rangle:=\bigcup_{n\in\bZ}R^qf_{\ast}((\omega_{Y/X}\otimes f^{\ast}(L^{\vee\vee})^{-1})^{-1}\otimes\Omega^{p}_{Y/X})(nD_f)$$
	be the sheaf of sections of $j_{X^o\ast}R^qf^o_{\ast}((B^o)^{-1}\otimes\Omega^{p}_{Y^o/X^o})$ that are meromorphic along $D_f$, where $j_{X^o}:X^o\to X$ is the immersion. Let
	$$R^qh^o_\ast(\Omega^{p}_{Z'^o/X'^o})\langle\rho^{-1}(D_g)\rangle:=\bigcup_{n\in\bZ}R^qh_\ast(\Omega^{p}_{Z'/X'})(n\rho^{-1}(D_g))$$
	ne the sheaf of sections of $j_{X'^o\ast}R^qh^o_\ast(\Omega^{p}_{Z'^o/X'^o})$ that are meromorphic along $\rho^{-1}(D_g)$, where $j_{X'^o}:X'^o\to X'$ is the immersion. (\ref{align_theta_vartheta_commute}) naturally extends to the diagram
	\begin{align}\label{align_iota}
		\xymatrix{
			R^qf^o_{\ast}((B^o)^{-1}\otimes\Omega^{p}_{Y^o/X^o})\langle D_f\rangle\ar[r]^-{\vartheta}\ar[d]^{\iota}& R^{q+1}f^o_{\ast}((B^o)^{-1}\otimes\Omega^{p-1}_{Y^o/X^o})\langle D_f\rangle\otimes\Omega_{X}\langle D_f\rangle\ar[d]^{\iota\otimes{\rm inclusion}}\\
			\psi_\ast \big(R^qh^o_\ast(\Omega^{p}_{Z'^o/X'^o})\langle\rho^{-1}(D_g)\rangle\big)\ar[r]^-{\theta}& \psi_\ast \big(R^{q+1}h^o_\ast(\Omega^{p-1}_{Z'^o/X'^o})\langle\rho^{-1}(D_g)\rangle\otimes\Omega_{X'}\langle\rho^{-1}(D_g)\rangle\big)
		}.
	\end{align}
	Define 
	\begin{align*}
		G^{p,q}:={\rm Im}(\iota)\cap \psi_{\ast}\left({_{<\psi^{-1}(D_f)_{\rm red}+E'}}H^{p,q}_{h^o}\right).
	\end{align*}
	Notice that the sections of $G^{p,q}$ have bounded degrees of poles along $X\backslash X_1$ since they lie in $\psi_{\ast}\left({_{<\psi^{-1}(D_f)_{\rm red}+E'}}H^{p,q}_{h^o}\right)$. Hence $G^{p,q}$ equals the intersection of $ \psi_{\ast}\left({_{<\psi^{-1}(D_f)_{\rm red}+E'}}H^{p,q}_{h^o}\right)$ with
	$${\rm Im}\left(R^qf_{\ast}((\omega_{Y/X}\otimes f^{\ast}(L^{\vee\vee})^{-1})^{-1}\otimes\Omega^{p}_{Y/X})(n_1D_f)\to \psi_\ast\big(R^qh_\ast(\Omega^{p}_{Z'/X'})(n_2\rho^{-1}(D_g))\big)\right)$$
	for some $n_1, n_2\in\bZ$. In particular, $G^{p,q}$ is a coherent sheaf on $X$ for every $p,q\geq0$.
	\begin{lem}\label{lem_thetaG}
		\begin{align}\label{align_ZV_is_log}
			\theta(G^{p,q})\subset G^{p-1,q+1}\otimes\Omega_X(\log D_f).
		\end{align}
	\end{lem}
	\begin{proof}
		\emph{Case I.} Let $x\in X^o=X\backslash D_f$ and let $z_1,\dots,z_d$ be holomorphic local coordinates at $x$. It suffices to show that
		\begin{align*}
			\theta(\frac{\partial}{\partial z_i})(G^{p,q})\subset G^{p-1,q+1},\quad \forall i=1,\dots,d.
		\end{align*}
		Let $v\in R^qf^o_{\ast}((B^o)^{-1}\otimes\Omega^{p}_{Y^o/X^o})$ such that $\iota(v)\in \psi_{\ast}\left({_{<\psi^{-1}(D_f)_{\rm red}+E'}}H^{p,q}_{h^o}\right)$. One has $$\vartheta(\frac{\partial}{\partial z_i})(v)\in R^{q+1}f^o_{\ast}((B^o)^{-1}\otimes\Omega^{p-1}_{Y^o/X^o})$$ according to (\ref{align_vartheta_nolog}). Thus
		\begin{align*}
			\theta(\frac{\partial}{\partial z_i})(\iota(v))=\iota\left(\vartheta(\frac{\partial}{\partial z_i})(v)\right)\in{\rm Im}(\iota),\quad \forall i=1,\dots,d
		\end{align*}
		by (\ref{align_iota}).
		Lemma \ref{lem_VZ_Xo} yields that $${\rm Im}(\iota)|_{X^o}\subset \psi_{\ast}\left({_{<\psi^{-1}(D_f)_{\rm red}+E'}}H^{p,q}_{h^o}\right)|_{X^o}.$$ This shows (\ref{align_ZV_is_log}) on $X\backslash D_f$.
		
		\emph{Case II.} Let $x\in D_f$. Let $z_1,\dots,z_d$ be holomorphic local coordinates at $x$ so that $D_f=\{z_1\cdots z_l=0\}$. Let
		\begin{align*}
			\xi_i:=\begin{cases}
				z_i\frac{\partial}{\partial z_i}, & i=1,\dots,l\\
				\frac{\partial}{\partial z_i}, & i=l+1,\dots,d
			\end{cases}.
		\end{align*}
		It suffices to show that 
		\begin{align*}
			\theta(\xi_i)(G^{p,q})\subset G^{p-1,q+1},\quad \forall i=1,\dots,d.
		\end{align*}
		Let $v\in R^qf^o_{\ast}((B^o)^{-1}\otimes\Omega^{p}_{Y^o/X^o})\langle D_f\rangle$ such that $\iota(v)\in \psi_{\ast}\left({_{<\psi^{-1}(D_f)_{\rm red}+E'}}H^{p,q}_{h^o}\right)$. It follows from (\ref{align_iota}) that 
		\begin{align*}
			\theta(\xi_i)(\iota(v))=\iota\left(\vartheta(\xi_i)(v)\right)\in{\rm Im}(\iota),\quad \forall i=1,\dots,d.
		\end{align*}
		Theorem \ref{thm_parabolic} yields that
		\begin{align}
			\theta(\xi_i)(\iota(v))\in \psi_{\ast}\left({_{<\psi^{-1}(D_f)_{\rm red}+E'}}H^{p,q}_{h^o}\right),\quad \forall i=1,\dots,d.
		\end{align}
		This shows (\ref{align_ZV_is_log}) on $X$.
	\end{proof}
	\subsubsection{Final proof}	
	It suffices to show the following lemma to finish the proof of Theorem \ref{thm_VZ_prolongation}, in accordance with Lemma \ref{lem_thetaG}.
	\begin{lem}
		There is a natural inclusion $\pi_\ast(\widetilde{L})\subset G^{n,0}$.
	\end{lem}
	\begin{proof}
		Consider the natural map
		\begin{align*}
			\widetilde{\alpha}:\pi_\ast(\widetilde{L})\to \pi_\ast\widetilde{f}_\ast(\widetilde{f}^\ast \widetilde{L})\simeq \pi_\ast\widetilde{f}_\ast(\widetilde{B}^{-1}\otimes\omega_{\widetilde{Y}/\widetilde{X}})\subset \pi_\ast g_\ast(\omega_{Z/\widetilde{X}})\subset \psi_{\ast}\left({_{<\psi^{-1}(D_f)_{\rm red}+E'}}H^{n,0}_{h^o}\right),
		\end{align*}
		where the last inclusion is deduced from Proposition \ref{prop_prolongation0_vs_geo}. Now it suffices to show that 
		\begin{align}\label{align_lem33}
			{\rm Im}(\widetilde{\alpha})|_{X^o}\subset{\rm Im}(\iota)|_{X^o}={\rm Im}(\iota_{X^o}).
		\end{align}
		Consider the natural map
		\begin{align*}
			L^{o}\to f^o_{\ast}(f^{o\ast} L^{o})\simeq f^o_{\ast}((B^o)^{-1}\otimes\Omega^{n}_{Y^o/X^o}). 
		\end{align*}
		Since $\psi_{\ast}\left({_{<\psi^{-1}(D_f)_{\rm red}+E'}}H^{n,0}_{h^o}\right)$ is torsion free, the composition map 
		$$\alpha: L^o\to f^o_\ast((B^o)^{-1}\otimes\Omega^{n}_{Y^o/X^o})\stackrel{\iota_{X^o}}{\to} \psi_{\ast}\left({_{<\psi^{-1}(D_f)_{\rm red}+E'}}H^{n,0}_{h^o}\right)|_{X^o}$$
		is injective. So it induces an injective morphism $$L^o\to{\rm Im}\left(\iota_{X^o}:f_\ast((B^o)^{-1}\otimes\Omega^{n}_{Y^o/X^o})\to \psi_{\ast}\left({_{<\psi^{-1}(D_f)_{\rm red}+E'}}H^{n,0}_{h^o}\right)|_{X^o}\right).$$
		According to (\ref{align_commut_s_s0}), one obtains that $\widetilde{\alpha}|_{X^o}=\alpha$. Consequently,
		$${\rm Im}(\widetilde{\alpha})|_{X^o}=\alpha(L^o)\subset {\rm Im}(\iota_{X^o}).$$
		The lemma is proved.
	\end{proof}
	\section{Admissible families of canonical stable minimal models}\label{section_boundedness}
	\subsection{Stable minimal models and their moduli}\label{section_moduli}
	We review the main results in \cite{Birkar2022} that will be used in the sequel. For the purpose of the present article, everything is defined over ${\rm Spec}(\bC)$. A \emph{stable minimal model} is a triple $(X,B),A$ where $X$ is a reduced connected projective scheme of finite type over ${\rm Spec}(\bC)$ and $A,B\geq0$ are $\bQ$-divisor satisfying the following conditions:
	\begin{itemize}
		\item $(X,B)$ is a projective connected slc pair,
		\item $K_X+B$ is semi-ample,
		\item $K_X+B+tA$ is ample for some $t>0$, and
		\item $(X,B+tA)$ is slc for some $t>0$.
	\end{itemize}
	Let
	$$d\in\bN,c\in\bQ^{\geq0},\Gamma\subset\bQ^{>0} \textrm{ a finite set, and }\sigma\in\bQ[t].$$
	A $(d,\Phi_c,\Gamma,\sigma)$-stable minimal model is a stable minimal model $(X,B),A$ satisfying the following conditions:
	\begin{itemize}
		\item $\dim X=d,$
		\item the coefficients of $A$ and $B$ are in $c\bZ^{\geq0}$,
		\item ${\rm vol}(A|_F)\in\Gamma$ where $F$ is any general fiber of the fibration $f:X\to Z$ determined by $K_X+B$, and
		\item ${\rm vol}(K_X+B+tA)=\sigma(t)$ for $0\leq t\ll 1$.
	\end{itemize}
	Let $S$ be a reduced scheme over ${\rm Spec}(\bC)$. A family of $(d,\Phi_c,\Gamma,\sigma)$-stable minimal models over $S$ consists of a projective morphism $X\to S$ of schemes and $\bQ$-divisors $A$ and $B$ on $X$ satisfying the following conditions:
	\begin{itemize}
		\item $(X,B+tA)\to S$ is a locally stable family (that is, $K_{X/S}+B+tA$ is $\bQ$-Cartier) for every sufficiently small rational number $t\geq0$,
		\item $A=cN$, $B=cD$ where $N, D\geq0$ are relative Mumford divisors, and
		\item $(X_s,B_s), A_s$ is a $(d,\Phi_c,\Gamma,\sigma)$-stable minimal model for each point $s\in S$.
	\end{itemize}
	Let ${\rm Sch}_{\bC}^{\rm red}$ denote the category of reduced schemes defined over ${\rm Spec}(\bC)$. Define 
	$$\sM^{\rm red}_{\rm slc}(d,\Phi_c,\Gamma,\sigma): S\mapsto\{\textrm{family of } (d,\Phi_c,\Gamma,\sigma)-\textrm{stable minimal models over }S\},$$
	the functor of groupoids over ${\rm Sch}_{\bC}^{\rm red}$.
	\begin{thm}[Birkar \cite{Birkar2022}]\label{thm_moduli_stable_var}
		There is a proper Deligne-Mumford stack $\sM_{\rm slc}(d,\Phi_c,\Gamma,\sigma)$ over $\bC$ such that the following hold.
		\begin{itemize}
			\item $\sM_{\rm slc}(d,\Phi_c,\Gamma,\sigma)|_{{\rm Sch}_{\bC}^{\rm red}}=\sM^{\rm red}_{\rm slc}(d,\Phi_c,\Gamma,\sigma)$ as functors of groupoids.
			\item $\sM_{\rm slc}(d,\Phi_c,\Gamma,\sigma)$ admits a  projective good coarse moduli space $M_{slc}(d,\Phi_c,\Gamma,\sigma)$.
		\end{itemize}
	\end{thm}
	\begin{proof}
		See the proof of \cite[Theorem 1.14]{Birkar2022}. Using the notations in \cite[\S 10.7]{Birkar2022}, we have $$\sM_{\rm slc}(d,\Phi_c,\Gamma,\sigma)=\left[M_{\rm slc}^e(d,\Phi_c,\Gamma,\sigma,a,r,\bP^n)/{\rm PGL}_{n+1}(\bC)\right],$$
		where the right hand side is the stacky quotient.
	\end{proof}
	A stable minimal model $(X,B),A$ is called a lc stable minimal model if $(X,B)$ is a lc pair.
	Let  $\sM_{\rm lc}(d,\Phi_c,\Gamma,\sigma)\subset \sM_{\rm slc}(d,\Phi_c,\Gamma,\sigma)$ denote the open substack that consists of  $(d,\Phi_c,\Gamma,\sigma)$-lc stable minimal models. We use $M_{\rm lc}(d,\Phi_c,\Gamma,\sigma)$ to denote the quasi-projective coarse moduli spaces of $\sM_{\rm lc}(d,\Phi_c,\Gamma,\sigma)$.
	\subsection{Polarization on $M_{\rm slc}(d,\Phi_c,\Gamma,\sigma)$}\label{section_polarization_moduli}
	In this subsection we consider some natural ample $\bQ$-line bundles on $M_{\rm slc}(d,\Phi_c,\Gamma,\sigma)$. Their constructions are implicit in the proof of \cite[Theorem 1.14]{Birkar2022}, depending on the ampleness criterion by Koll\'ar \cite{Kollar1990}. 
	Fix a data $d,\Phi_c,\Gamma,\sigma$. Since $\sM_{\rm slc}(d,\Phi_c,\Gamma,\sigma)$ is of finite type, there are constants $$(a,r,j)\in\bQ^{\geq0}\times(\bZ^{>0})^{2},$$ depending only on $d,\Phi_c,\Gamma,\sigma$ such that  every $(d,\Phi_c,\Gamma,\sigma)$-stable minimal model $(X,B),A$ satisfies the following conditions (c.f. \cite[Lemma 10.2]{Birkar2022}):
	\begin{itemize}
		\item $X+B+aA$ is slc,
		\item $r(K_X+B+aA)$ is a very ample integral Cartier divisor with 
		$$H^i(X,kr(K_X+B+aA))=0,\quad\forall i>0,\forall k>0,$$
		\item the embedding $X\hookrightarrow\bP(H^0(X,r(K_X+B+aA)))$ is defined by degree $\leq j$ equations, and
		\item the multiplication map $$S^j(H^0(X,r(K_X+B+aA)))\to H^0(X,jr(K_X+B+aA))$$ is surjective.
	\end{itemize}
	\begin{defn}
		$(a,r,j)\in\bQ^{\geq0}\times(\bZ^{>0})^{2}$ that satisfies the conditions above is called a \emph{ $(d,\Phi_c,\Gamma,\sigma)$-polarization data}.
	\end{defn}
	Let $(a,r,j)\in\bQ^{\geq0}\times(\bZ^{>0})^{2}$ be a $(d,\Phi_c,\Gamma,\sigma)$-polarization data.
	Let $(X,B),A\to S$ be a family of $(d,\Phi_c,\Gamma,\sigma)$-stable minimal models. Then $f_\ast(r(K_{X/S}+B+aA))$ is locally free and commutes with an arbitrary base change. Therefore the assignment
	$$f:(X,B),A\to S\in\sM_{\rm slc}(d,\Phi_c,\Gamma,\sigma)(S)\mapsto f_\ast(r(K_{X/S}+B+aA))$$
	gives a locally free coherent sheaf on the stack $\sM_{\rm slc}(d,\Phi_c,\Gamma,\sigma)$, denoted by $\Lambda_{a,r}$. Let $\lambda_{a,r}:=\det(\Lambda_{a,r})$. Since $\sM_{\rm slc}(d,\Phi_c,\Gamma,\sigma)$ is Deligne-Mumford, some power $\lambda_{a,r}^{\otimes k}$ descends to a line bundle on $M_{\rm slc}(d,\Phi_c,\Gamma,\sigma)$. For this reason we regard $\lambda_{a,r}$ as a $\bQ$-line bundle on $M_{\rm slc}(d,\Phi_c,\Gamma,\sigma)$. 
	\begin{prop}\label{prop_ample_line_bundle_moduli}
		Let $(a,r,j)\in\bQ^{\geq0}\times(\bZ^{>0})^{2}$ be a $(d,\Phi_c,\Gamma,\sigma)$-polarization data. Then $\lambda_{a,r}$ is ample on $M_{\rm slc}(d,\Phi_c,\Gamma,\sigma)$.
	\end{prop}
	\begin{proof}
		By the same arguments as in \cite[\S 2.9]{Kollar1990}, it suffices to show that $f_\ast(r(K_{X/S}+B+aA))$ is nef when $S$ is a smooth projective curve. This has been accomplished by Fujino \cite{Fujino2018} and Kov\'acs-Patakfalvi \cite{Kovacs2017}.
	\end{proof}
	\subsection{A technical lemma for semistable families}\label{section_ssred}
	A morphism $f:Y\to X$ between smooth varieties is called \emph{semistable} (resp. \emph{strictly semistable}) if there is a (not necessarily connected) smooth divisor $D$ on $X$ such that the following conditions hold.
	\begin{enumerate}
		\item $f$ is a submersion over $X\backslash D$ and the schematic preimage $f^{-1}(D)$ is a (resp. reduced) simple normal crossing divisor on $Y$. 
		\item $f$ sends submersively any
		stratum of $f^{-1}(D)_{\rm red}$ onto an irreducible component of $D$.
	\end{enumerate}
	A morphism $f:Y\to X$ between smooth varieties is \emph{strictly semistable in codimension one} if there is a dense Zariski open subset $U\subset X$ with ${\rm codim}_X(X\backslash U)\geq 2$, such that $f|_{f^{-1}(U)}:f^{-1}(U)\to U$ is strictly semistable.
	
	For a surjective morphism $Y\to X$ between algebraic varieties, let $Y^{[r]}_X$ denote the main component of the $r$-fiber product $Y\times_XY\times_X\cdots\times_XY$ (that is, the union of irreducible components that is mapped onto $X$). Let $f^{[r]}:Y^{[r]}_X\to X$ denote the projection map. The following lemma is known to experts. We present the proof for the convenience of readers.
	\begin{lem}\label{lem_mild_pushforward}
		Let $f:Y\to X$ be a strictly semistable morphism and
		$\tau:Y^{(r)}\to Y^{[r]}_X$ a desingularization. Denote $f^{(r)}=f^{[r]}\tau$. Then the following hold.
		\begin{enumerate}
			\item $\tau_{\ast}(\omega_{Y^{(r)}}^{\otimes k})\simeq\omega_{Y^{[r]}_X}^{\otimes k}$ for every $k\geq 1$, where $\omega_{Y^{[r]}_X}$ is the dualizing sheaf (invertible since $Y^{[r]}_X$ is Gorenstein).
			\item $f^{(r)}_\ast(\omega_{Y^{(r)}/X}^{\otimes k})$ is a reflexive sheaf for every $k\geq 1$.
			\item $f^{(r)}_\ast(\omega_{Y^{(r)}/X}^{\otimes k})\simeq (f_\ast(\omega_{Y/X}^{\otimes k})^{\otimes r})^{\vee\vee}$ for every $k\geq 1$.
		\end{enumerate}
	\end{lem}
	\begin{proof}
		A strictly semistable morphism is weakly semistable in the sense of Abramovich-Karu \cite{Abramovich2000}. Hence $Y^{[r]}_X$ has only normal, rational and Gorenstein singularities by \cite[Proposition 6.4]{Abramovich2000}. Consequently, it can be inferred that $Y^{[r]}_X$ has canonical singularities. The first claim is proved. 
		
		For the second claim, it suffices to show that any section of $f^{[r]}_\ast(\omega_{Y^{[r]}_X/X}^{\otimes k})\simeq f^{(r)}_\ast(\omega_{Y^{(r)}/X}^{\otimes k})$ extends across an arbitrary locus of codimension $\geq 2$ . Let $U\subset X$ be an open subset and $Z\subset U$ a Zariski closed subset of codimension $\geq 2$. Let $$s\in \Gamma(U\backslash Z, f^{[r]}_\ast(\omega_{Y^{[r]}_X/X}^{\otimes k}))=\Gamma((f^{[r]})^{-1}(U\backslash Z), \omega_{Y^{[r]}_X/X}^{\otimes k}).$$
		Since $f$ is flat, so is $f^{[r]}$. Hence $(f^{[r]})^{-1}(Z)$ is of codimension $\geq 2$ in $(f^{[r]})^{-1}(U)$. Since $Y^{[r]}_X$ is normal and $\omega_{Y^{[r]}_X/X}^{\otimes k}$ is invertible, there is 
		$$\widetilde{s}\in \Gamma(U, f^{[r]}_\ast(\omega_{Y^{[r]}_X/X}^{\otimes k}))=\Gamma((f^{[r]})^{-1}(U), \omega_{Y^{[r]}_X/X}^{\otimes k})$$
		that extends $s$. This proves Claim (2).
		
		Finally we show the last claim. Since $f^{[r]}$ and $f$ are Gorenstein, one obtains that 
		$$\omega_{Y^{[r]}_X/X}^{\otimes k}\simeq\otimes_{i=1}^r p_i^\ast\omega_{Y/X}^{\otimes k}$$
		where $p_i:Y^{[r]}_X\to Y$ denotes the projection to the $i$th component. Let $U\subset X$ be the largest open subset on which $f^{[r]}_\ast(\omega_{Y^{[r]}_X/X}^{\otimes k})$ and $f_\ast(\omega_{Y^/X}^{\otimes k})$ are locally free. Since the relevant sheaves are torsion free, $X\backslash U$ is of  codimension $\geq 2$. By the flat base change we obtain that
		$$f^{(r)}_\ast(\omega_{Y^{(r)}/X}^{\otimes k})|_U\simeq f_\ast(\omega_{Y/X}^{\otimes k})^{\otimes r}|_U.$$
		Since $f^{(r)}_\ast(\omega_{Y^{(r)}/X}^{\otimes k})$ and $ (f_\ast(\omega_{Y/X}^{\otimes k})^{\otimes r})^{\vee\vee}$ are reflexive, we have proven Claim (3).
	\end{proof}
	\subsection{Higgs sheaves associated to a family of lc stable minimal models}\label{section_VZHiggs_stable_family}
	The aim of this section is to prove Theorem \ref{thm_main_VZHiggs} (=Theorem \ref{thm_big_Higgs_sheaf}).
	Recall that a family $f:(X,\Delta)\to S$ is \emph{log smooth} if $f$ is a smooth projective morphism between varieties and $\Delta$ is a simple normal crossing $\bQ$-divisor on $X$ such that every stratum of ${\rm supp}(\Delta)$ is smooth over $S$.
	\begin{defn}\label{defn_admissible}
		Let $f:(X,B),A\to S$ be a family of stable minimal models over a variety $S$. A \emph{log smooth birational model} of $f$ is a log smooth family $(X',\Delta')\to S$, together with a birational map $g:X'\dashrightarrow X$ over $S$ such that the following conditions hold.
		\begin{enumerate}
			\item $g$ is defined on a dense Zariski open subset of ${\rm supp}(\Delta')$ and $A+B$ is the birational transform of $\Delta'$.
			\item For every $s\in S(\bC)$, $g$ is defined on Zariski open subsets of $X'_s$ and ${\rm supp}(\Delta'_s)$, and $g|_{X'_s}:X'_s\dashrightarrow X_s$ is a birational map.
		\end{enumerate}
	    If $g$ only satisfies Condition (1), we say that $(X',\Delta')\to S$ is a \emph{log smooth weakly birational model} of $f$.
		$f$ is called \emph{admissible} (resp. \emph{weakly admissible}) if it admits a log smooth (resp. weakly) birational model and the coefficients of $B$ lie in $[0,1)$.
	\end{defn}
	\begin{lem}[Relative Kawamata's covering]\label{lem_rel_Kawamata_covering}
		Let $f:X\to S$ be a morphism between smooth projective varieties and $D$ a simple normal crossing $\bQ$-divisor on $X$ whose coefficients lie in $[0,1)$. Let $S^o\subset S$ be a Zariski open subset such that $(f^{-1}(S^o),D|_{f^{-1}(S^o)})\to S^o$ is a log smooth family. Then there is a finite surjective morphism $h:Y\to X$ satisfying the following conditions:
		\begin{enumerate}
			\item $Y$ is a smooth projective variety,
			\item $f\circ h$ is a smooth morphism over $S^o$, and
			\item there is a $\bQ$-divisor $F\geq0$ on $Y$ such that 
			\begin{align}\label{align_canonical_bundle_formula}
				h^\ast(K_X+D)=K_Y-F.
			\end{align}
		\end{enumerate}
	\end{lem}
	\begin{proof}
		The proof is the same as \cite[Theorem 17]{Kawamata1981} except two modifications. The first is that the general hyperplanes $H_1,\dots,H_d$ in loc. cit.  should satisfy that $$H_{1}+\dots+H_{d}+ D|_{f^{-1}(S^o)}$$ is a relative simple normal crossing divisor over $S^o$. The second is that one should let $m_i$ in loc. cit. be sufficiently large in order to ensure the validity of (\ref{align_canonical_bundle_formula}).
	\end{proof}
	\begin{thm}\label{thm_big_Higgs_sheaf}
		Let $f^o:(X^o,B^o),A^o\to S^o$ be a weakly admissible family of $(d,\Phi_c,\Gamma,\sigma)$-lc stable minimal models over a smooth quasi-projective variety $S^o$ which defines a generically finite morphism $\xi^o:S^o\to M_{\rm lc}(d,\Phi_c,\Gamma,\sigma)$.
		Let $S$ be a smooth projective variety containing $S^o$ as a Zariski open subset such that $D:=S\backslash S^o$ is a (reduced) simple normal crossing divisor and $\xi^o$ extends to a morphism $\xi:S\to M_{\rm lc}(d,\Phi_c,\Gamma,\sigma)$. Let $\sL$ be a line bundle on $S$. Then there exist the following data.
		\begin{enumerate}
			\item A projective birational morphism $\pi:S'\to S$ such that $S'$ is smooth, $\pi^{-1}(D)$ is a simple normal crossing divisor and $\pi$ is a composition of smooth blow-ups.
			\item A (possibly non-reduced) effective exceptional divisor $E$ of $\pi$ such that $E\cup\pi^{-1}(D)$ has a simple normal crossing support.
			\item A $\bQ$-polarized variation of Hodge structure of weight $w>0$ on $S'\backslash (E\cup \pi^{-1}(D))$ with $(H=\bigoplus_{p+q=w} H^{p,q},\theta,h)$ its associated Higgs bundle by taking the total graded quotients with respect to the Hodge filtration. Here $h$ is the Hodge metric.
		\end{enumerate}
		These data satisfy the following conditions.
		\begin{enumerate}
			\item There is a coherent ideal sheaf $I_Z$ on $S$ whose co-support $Z$ is contained in $D$ and ${\rm codim}_S(Z)\geq2$, and a natural inclusion $\sL\otimes I_Z\subset \pi_{\ast}\left({_{<\pi^{-1}(D)+E}}H^{w,0}\right)$.
			\item Let $(\bigoplus_{p=0}^w L^p,\theta)$ be the log Higgs subsheaf generated by $L^0:=\sL\otimes I_Z$, where $$L^p\subset\pi_{\ast}\left({_{<\pi^{-1}(D)+E}}H^{w-p,p}\right).$$ Then the Higgs field $$\theta:L^p|_{S\backslash (D\cup\pi(E))}\to L^{p+1}|_{S\backslash (D\cup\pi(E))}\otimes \Omega_{S\backslash (D\cup\pi(E))}$$ is holomorphic over $S\backslash D$ for each $0\leq p<n$, that is, 
			$$\theta(L^p)\subset L^{p+1}\otimes\Omega_S(\log D).$$
		\end{enumerate}
	\end{thm}
	\begin{proof}
		\emph{Step 1: Compactify the family.}
		By the properness of $\sM_{\rm slc}(d,\Phi_c,\Gamma,\sigma)$, we have the following constructions.
		\begin{itemize}
			\item A generically finite proper surjective morphism $\sigma:\widetilde{S}\to S$ from a smooth projective variety $\widetilde{S}$ such that $\sigma^{-1}(D)$ is a simple normal crossing divisor. $\sigma$ is a combination of smooth blow-ups and a finite flat morphism. 
			\item Let $\widetilde{S}^o:=\sigma^{-1}(S^o)$ and  $\widetilde{X}^o:=\widetilde{S}^o\times_{S^o}X^o$. Let $\widetilde{A}^o$ and $\widetilde{B}^o$ be the divisorial pullbacks of $A^o$ and $B^o$ on $\widetilde{X}^o$ respectively. There is a completion $\widetilde{f}:(\widetilde{X},\widetilde{B}),\widetilde{A}\to\widetilde{S}$ of the base change family $(\widetilde{X}^o,\widetilde{B}^o),\widetilde{A}^o\to \widetilde{S}^o$ such that 
			$\widetilde{f}\in\sM_{\rm slc}(d,\Phi_c,\Gamma,\sigma)(\widetilde{S})$.
		\end{itemize}
		\emph{Step 2: Take the log smooth weakly birational models.} 
		Consider the following diagram
		\begin{align*}
			\xymatrix{
				X'^o\ar@{-->}[d]^{\rho^o}\ar@/_/[dd]_-{f'^o} &\widetilde{X}'^o \ar@/_/[dd]_-{\widetilde{f}'^o} \ar[l]\ar@{-->}[d]^{\widetilde{\rho}^o} \ar[r]^{\subset} & \widetilde{X}'\ar@{-->}[d]^{\widetilde{\rho}}\ar@/_/[dd]_-{\widetilde{f}'}\\
				X^o\ar[d]^{f^o} &\widetilde{X}^o \ar[l]\ar[d] \ar[r]_{\subset} & \widetilde{X}\ar[d]^{\widetilde{f}}\\
				S^o& \widetilde{S}^o\ar[l]\ar[r]^{\subset} & \widetilde{S}
			}
		\end{align*}
		where the arrows are explained as follows.
		\begin{itemize}
			\item $f'^o:X'^o\to S^o$ is a projective smooth morphism and $\rho^o:X'^o\dashrightarrow X^o$ is a birational map over $S^o$ whose image contains a dense Zariski open subset of ${\rm supp}(A^o+B^o)$. Let the $\bQ$-divisors $A'^o$ and $B'^o$ on $X'^o$ be the birational transforms of $A^o$ and $B^o$ respectively. $f'^o:(X'^o,A'^o+B'^o)\to S^o$ is a log smooth weakly birational model of $f^o:(X^o,B^o),A^o\to S^o$.
			\item $\widetilde{f}'^o$ and $\widetilde{\rho}^o$ are the base changes of $f'^o$ and $\rho^o$ respectively. Let $\widetilde{A}'^o$ and $\widetilde{B}'^o$ be the birational transforms of $\widetilde{A}^o$ and $\widetilde{B}^o$ on $\widetilde{X}'^o$ respectively. Then $\widetilde{f}'^o:(\widetilde{X}'^o,\widetilde{A}'^o+\widetilde{B}'^o)\to \widetilde{S}^o$ is a log smooth weakly birational model of $(\widetilde{X}^o,\widetilde{B}^o),\widetilde{A}^o\to \widetilde{S}^o$.
			\item $\widetilde{f}':\widetilde{X}'\to \widetilde{S}$ is a completion of $\widetilde{f}'^o$. Since the supports of $\widetilde{A}$ and $\widetilde{B}$ do not contain any irreducible component of any fiber of $\widetilde{f}$, $\widetilde{\rho}^o$ extends naturally to a birational map $\widetilde{\rho}:\widetilde{X}'\dashrightarrow\widetilde{X}$ whose image contains a dense Zariski open subset of ${\rm supp}(\widetilde{A}+\widetilde{B})$. Hence we can define the birational transforms of $\widetilde{A}$ and $\widetilde{B}$ on $\widetilde{X}'$, denoted by $\widetilde{A}'$ and $\widetilde{B}'$ respectively. By replacing $\widetilde{X}'$ with a suitable blow-up on $\widetilde{X}'$ whose center lies in $\widetilde{f}'^{-1}\sigma^{-1}(D)$, it may be assumed that the following are valid:
			\begin{itemize}
				\item $\widetilde{X}'$ is a smooth projective variety,
				\item $\widetilde{f}'^{-1}\sigma^{-1}(D)+ \widetilde{A}'+\widetilde{B}'$ is a simple normal crossing $\bQ$-divisor, and
				\item $\widetilde{f}':\widetilde{X}'\to \widetilde{S}$ is smooth over $\widetilde{S}^o$.
			\end{itemize}
		\end{itemize}
		The constructions yield that 
		\begin{align}\label{align_NK}
			\widetilde{f}_{\ast}(r(K_{\widetilde{X}/\widetilde{S}}+\widetilde{B}+a\widetilde{A}))\simeq \widetilde{f}'_{\ast}(r(K_{\widetilde{X}'/\widetilde{S}}+\widetilde{B}'+a\widetilde{A}'))
		\end{align}
		whenever $r(K_{\widetilde{X}/\widetilde{S}}+\widetilde{B}+a\widetilde{A})$ is integral.
		
		\emph{Step 3: Kawamata's trick.} 
		Let $(a,r,j)\in\bQ^{\geq0}\times(\bZ^{>0})^{2}$ be a $(d,\Phi_c,\Gamma,\sigma)$-polarization data with $0<a\ll 1$ so that $\widetilde{B}'+a\widetilde{A}'$ is a simple normal crossing divisor with coefficients in $[0,1)$. Lemma \ref{lem_rel_Kawamata_covering} yields that there is a finite surjective morphism $\varrho:\widetilde{Y}'\to \widetilde{X}'$ satisfying the following conditions:
		\begin{enumerate}
			\item $\widetilde{Y}'$ is a smooth projective variety,
			\item $\widetilde{f}'\varrho$ is smooth over $S^o$, and
			\item there is a $\bQ$-divisor $F\geq0$ such that 
			\begin{align}\label{align_adjunction_to_smooth}
				\varrho^\ast(K_{\widetilde{X}'}+\widetilde{B}'+a\widetilde{A}')=K_{\widetilde{Y}'}-F.
			\end{align}
		\end{enumerate}
		The same construction is valid on $X'^o$. One can choose a suitable Kawamata covering map $Y^o\to X'^o$ such that $\widetilde{f}'\varrho:\widetilde{Y}'\to \widetilde{S}$ is a completion of the base change morphism $Y^o\times_{S^o}\widetilde{S}^o\to\widetilde{S}^o$. 
		Consider the following diagram
		\begin{align}
			\xymatrix{
				\widetilde{Y}'\ar[dr]_{\widetilde{f}'\varrho}&\widetilde{Y}\ar[r]\ar[l] & Y\ar[d]\\
				&\widetilde{S}\ar[r]^{\sigma} & S
			},
		\end{align}
		where $Y\to S$ is a completion of $Y^o\to S^o$ with $Y$ smooth and projective, $\widetilde{Y}\to\widetilde{Y}'$ is some modification,  biholomorphic over $Y^o\times_{S^o}\widetilde{S}^o$, so that there is a morphism $\widetilde{Y}\to Y$ making the diagram commutative. We may assume that $\sigma$ is sufficiently ramified so that $\widetilde{Y}\to\widetilde{S}$ is strictly semistable in codimension one.
		Let $g:Y\to S$ and $\widetilde{g}:\widetilde{Y}\to\widetilde{S}$ be the maps that we have just constructed. According to (\ref{align_NK}) and (\ref{align_adjunction_to_smooth}), it follows that there is an inclusion
		\begin{align}\label{align_put_polarization_to_pushforward1}
			\widetilde{f}_\ast(rK_{\widetilde{X}/\widetilde{S}}+r\widetilde{B}+ra\widetilde{A})\subset \widetilde{g}_\ast(rK_{\widetilde{Y}/\widetilde{S}}).
		\end{align}
		
		\emph{Last step.} We finish the proof by applying Theorem \ref{thm_VZ_prolongation}. Let
		$$W:=\widetilde{f}_\ast(rK_{\widetilde{X}/\widetilde{S}}+r\widetilde{B}+ra\widetilde{A})\quad\textrm{and}\quad l:={\rm rank}(W).$$
		By the construction of $\lambda_{a,r}$ one has
		\begin{align*}
			(\xi\circ\sigma)^\ast\lambda_{a,r}\simeq\det(W).
		\end{align*}
		Since $\lambda_{a,r}$ is ample (Proposition \ref{prop_ample_line_bundle_moduli}) and $\xi\circ\sigma$ is generically finite, $\det(W)$ is big.
		Take an ample line bundle $M$ on $S$ so that there is an inclusion
		$$\sigma_{\ast}\sO_{\widetilde{S}}\otimes M^{-1}\subset\sO_S^{\oplus N}$$
		for some $N>0$. Since $\det(W)$ is big, there is an inclusion $$\sigma^\ast(\sL^{\otimes r}\otimes M)\subset \det(W)^{\otimes kr}$$ for some $k>0$.
		Let $\widetilde{Y}^{(klr)}$ denote  a functorial desingularization of the main component of the $klr$-fiber product $\widetilde{Y}\times_{\widetilde{S}}\times\cdots\times_{\widetilde{S}}\widetilde{Y}$ and let $\widetilde{g}^{(klr)}:\widetilde{Y}^{(klr)}\to \widetilde{S}$ denote the projection map. Define $g^{(klr)}:Y^{(klr)}\to S$ similarly. We may assume that there is a morphism $\widetilde{Y}^{(klr)}\to Y^{(klr)}$ such that the diagram
		\begin{align*}
			\xymatrix{
				\widetilde{Y}^{(klr)}\ar[r]\ar[d]^{\widetilde{g}^{(klr)}} & Y^{(klr)} \ar[d]^{g^{(klr)}}\\
				\widetilde{S}\ar[r]^{\sigma}&S
			}
		\end{align*}
		is 	commutative.
		According to (\ref{align_put_polarization_to_pushforward1}), there is an inclusion
		$W\subset \widetilde{g}_\ast(rK_{\widetilde{Y}/\widetilde{S}})$. 
		Since $\widetilde{Y}\to \widetilde{S}$ is smooth over $\widetilde{S}^o$ and is strictly semistable in codimension one, it follows from Lemma \ref{lem_mild_pushforward} that there is a natural inclusion
		\begin{align*}
			\det(W)^{\otimes kr}\otimes I_{\widetilde{Z}}\to \left(\widetilde{g}_{\ast}(rK_{\widetilde{Y}/\widetilde{S}})^{\otimes klr}\right)^{\vee\vee}\otimes I_{\widetilde{Z}}\subset\widetilde{g}^{(klr)}_{\ast}(rK_{\widetilde{Y}^{(klr)}/\widetilde{S}}),
		\end{align*}
		where $I_{\widetilde{Z}}$ is an ideal sheaf whose co-support $\widetilde{Z}$ lies in $\sigma^{-1}(D)$ and ${\rm codim}_{\widetilde{S}}(\widetilde{Z})\geq 2$. 
		Notice that there is an inclusion
		\begin{align*}
			\widetilde{g}^{(klr)}_\ast(rK_{\widetilde{Y}^{(klr)}/\widetilde{S}})\subset\sigma^\ast g_\ast^{(klr)}(rK_{Y^{(klr)}/S})
		\end{align*}
		by \cite[Lemma 3.2]{Viehweg1983} (see also \cite[Lemma 3.1.20]{Fujino2020}). 
		Let $I$ be an ideal sheaf on $S$ such that the map $\sigma^\ast(I)\to\sO_{\widetilde{S}}$ factors through $I_{\widetilde{Z}}$.
		Taking the composition of the maps above, we get a morphism
		$$\sigma^\ast(\sL^{\otimes r}\otimes M\otimes I^{\otimes r})\to\sigma^\ast g_\ast^{(klr)}(rK_{Y^{(klr)}/S}).$$
		This induces a non-zero map
		\begin{align*}
			\sL^{\otimes r}\otimes I^{\otimes r}\to g_\ast^{(klr)}(rK_{Y^{(klr)}/S})\otimes\sigma_\ast\sO_{\widetilde{S}}\otimes M^{-1}\subset g_\ast^{(klr)}(rK_{Y^{(klr)}/S})^{\oplus N}.
		\end{align*}
		Hence we obtain a non-zero map
		\begin{align}\label{align_send_A_in1}
			\sL^{\otimes r}\otimes I^{\otimes r}\to g_\ast^{(klr)}(rK_{Y^{(klr)}/S}).
		\end{align}
		Applying Theorem \ref{thm_VZ_prolongation} to the morphism $Y^{(klr)}\to S$ (which is smooth over $S^o$) and the torsion free sheaf $\sL\otimes I$, we obtain the theorem.
	\end{proof}
	\section{Hyperbolicity properties for admissible families of lc stable minimal models}
	Based on the constructions in the previous sections (especially Theorem \ref{thm_big_Higgs_sheaf}), we investigate various hyperbolicity properties of the base of a (weakly) admissible family of lc stable minimal models. 
	\subsection{Viehweg hyperbolicity}
	\begin{thm}\label{thm_VH}
		Let $f^o:(X^o,B^o),A^o\to S^o$ be a weakly admissible family of $(d,\Phi_c,\Gamma,\sigma)$-lc stable minimal models over a smooth quasi-projective variety $S^o$ which defines a generically finite morphism $\xi^o:S^o\to M_{\rm lc}(d,\Phi_c,\Gamma,\sigma)$.
		Let $S$ be a smooth projective variety containing $S^o$ as a Zariski open subset such that $D:=S\backslash S^o$ is a divisor. Then $\omega_S(D)$ is big.
	\end{thm}
	\begin{proof}
		Since the claim of the theorem is independent of the choice of the compactification $S^o\subset S$, we may assume that $D$ is a simple normal crossing divisor and the morphism $\xi^o$ extends to a morphism $\xi:S\to M_{\rm slc}(d,\Phi_c,\Gamma,\sigma)$.
		Take a line bundle $\sL$ on $S$ so that $\sL\otimes\sO_S(-D)$ is big.
		Let $$\bigoplus_{p=0}^w L^p\subset \pi_{\ast}\left({_{<\pi^{-1}(D)+E}}H^{w}_{h^o}\right)=\bigoplus_{p=0}^w\pi_{\ast}\left({_{<\pi^{-1}(D)+E}}H^{w-p,p}_{h^o}\right)$$ be the Higgs subsheaf generated by $L^0=\sL\otimes I_Z$ as in Theorem \ref{thm_big_Higgs_sheaf}, where $I_Z$ is an ideal sheaf on $S$ whose co-support $Z$ lies in $D$ and ${\rm codim}_S(Z)\geq 2$. Then $$\sL\otimes\sO_S(-D)\otimes I_Z\subset\pi_{\ast}\left({_{<E}}H^{w,0}_{h^o}\right)$$ and $\bigoplus_{p=0}^w L^p\otimes\sO_S(-D)$ is a log Higgs subsheaf of $\pi_{\ast}\left({_{<E}}H^{w}_{h^o}\right)$ such that
		\begin{align*}
			\theta(L^p\otimes\sO_S(-D))\subset L^{p+1}\otimes\sO_S(-D)\otimes\Omega_S(\log D),\quad\forall p=0,\dots,w-1.
		\end{align*}	
		Consider the diagram
		\begin{align*}
			L^0\otimes\sO_S(-D)\stackrel{\theta}{\to}L^1\otimes\sO_S(-D)\otimes\Omega_S(\log D)\stackrel{\theta\otimes{\rm Id}}{\to}L^2\otimes\sO_S(-D)\otimes\Omega^{\otimes 2}_S(\log D)\to\cdots.
		\end{align*}
		Notice that there is a minimal $n_0\leq w$ such that $L^0\otimes\sO_S(-D)$ is sent into
		$$\ker\left(L^{n_0}\otimes\sO_S(-D)\otimes\Omega^{\otimes n_0}_S(\log D)\to L^{n_0+1}\otimes\sO_S(-D)\otimes\Omega^{\otimes n_0+1}_S(\log D)\right)\subset K\otimes \Omega^{\otimes n_0}_S(\log D)$$
		where 
		$$K=\ker\left(\theta:\pi_{\ast}\left({_{<E}}H^{w}_{h^o}\right)\to \pi_{\ast}\left({_{<E}}H^{w}_{h^o}\otimes\Omega_{\widetilde{S}}(\log \pi^{-1}(D)\cup E)\right)\right).$$
		Since $n_0$ is minimal and $K$ is torsion free, we obtain an inclusion
		\begin{align}
			\sL\otimes\sO_S(-D)\otimes I_Z\subset K\otimes \Omega^{\otimes n_0}_S(\log D).
		\end{align}
		This induces a nonzero morphism
		\begin{align}\label{align_send_L_to_Higgsker}
			\beta:\sL\otimes\sO_S(-D)\otimes I_Z\otimes K^{\vee}\to \Omega^{\otimes n_0}_S(\log D).
		\end{align}
		Since $K\subset \pi_{\ast}\left({_{<E}}H^{w}_{h^o}\right)$,  $K^{\vee}$ is weakly positive by Proposition \ref{prop_semipositive_kernel}. Since $\sL$ and $\sL\otimes I_Z$ are isomorphic in codimension one, $\sL\otimes\sO_S(-D)\otimes I_Z$ is big by assumption.
		Consequently, $\Omega^{\otimes n_0}_S(\log D)$ contains the big sheaf ${\rm Im}(\beta)$, which implies that $n_0>0$. Thus  $\omega_S(D)$ is big by \cite[Theorem 7.11]{CP2019}.
	\end{proof}
	\subsection{Big Picard theorem}
	\subsubsection{DLSZ criterion}
	In this section we review the criterion by Deng-Lu-Sun-Zuo \cite{DLSZ} on the big Picard type result via Finsler pseudometrics.
	\begin{defn}[Finsler metric]
		Let $E$ be a holomorphic vector bundle on a complex manifold
		$X$. A Finsler pseudometric on $E$ is a continuous function $h:E\to[0,\infty)$ such that
		$$h(av)=|a|h(v),\quad \forall a\in \bC, \forall v\in E.$$
		We call $h$ a Finsler metric if it is non-degenerate, that is, if $h(v)=0$ only
		when $v=0$ holds.
	\end{defn}
	\begin{thm}{\cite[Theorem A]{DLSZ}}\label{thm_big_Picard_LSZ}
		Let $(X,\omega)$ be a compact K\"ahler manifold and $D$ a simple normal crossing divisor on $X$. Let $\gamma:\Delta^\ast:=\{z\in\bC\mid 0<|z|<1\}\to X\backslash D$ be a holomorphic
		map. Assume that there is a Finsler pseudometric $h$ on $T_X(-\log D)$ such that $|\frac{\partial}{\partial z}|^2_{\gamma^\ast h}$ is not identically zero and that the following inequality holds in the sense of currents
		$$\partial\dbar\log\left|\frac{\partial}{\partial z}\right|^2_{\gamma^\ast h}\geq\gamma^\ast(\omega).$$
		Then $\gamma$ extends to a holomorphic map $\overline{\gamma}:\Delta\to X$.
	\end{thm}
	\subsubsection{Picard pair}
	Let $X$ be a compact complex space and $Z\subset X$ a closed analytic subset. $(X,Z)$ is called a \emph{Picard pair} if either $0\in\overline{\gamma^{-1}(S\backslash S^o)}$ or $\gamma$ can be extended to a holomorphic map $\overline{\gamma}:\Delta\to S$ for any given holomorphic map $\gamma:\Delta^\ast\to S$ from the punctured unit disc $\Delta^\ast$. In particular, any holomorphic map $\gamma:\Delta^\ast\to S^o$ extends to a holomorphic map $\overline{\gamma}:\Delta\to S$. The classical big Picard theorem is equivalent to that $(\bP^1,\{0,1,\infty\})$ is a Picard pair.
	\begin{lem}\label{lem_Picardpair_modification}
		Let $X$ be a compact complex space and $Z\subset X$ a closed analytic subset. Let $\pi:X'\to X$ be a proper bimeromorphic morphism which is biholomorphic over $X\backslash Z$. Then $(X,Z)$ is a Picard pair if and only if $(X',\pi^{-1}(Z))$ is a Picard pair.
	\end{lem}
	\begin{proof}
		Assume that $(X,Z)$ is a Picard pair and $\gamma:\Delta^\ast\to X'$ is a holomorphic map such that $0\notin\overline{\gamma^{-1}(\pi^{-1}(Z))}$. Then $\pi\circ\gamma:\Delta^\ast\to X$ extends to a holomorphic morphism $\overline{\pi\circ\gamma}:\Delta\to X$. Since $\pi$ is biholomorphic over $X\backslash Z$ and there is a neighborhood $U\subset\Delta$ of $0$ such that $\pi\circ\gamma(U\backslash\{0\})\subset X\backslash Z$, the map $\overline{\pi\circ\gamma}:\Delta\to X$ can be lifted to a holomorphic morphism $\Delta\to X'$. This proves that $(X',\pi^{-1}(Z))$ is a Picard pair. 
		
		Conversely, we assume that $(X',\pi^{-1}(Z))$ is a Picard pair. Let $\gamma:\Delta^\ast\to X$ be a holomorphic map such that $0\notin\overline{\gamma^{-1}(Z)}$. Take a neighborhood $U\subset\Delta$ of $0$ such that $\gamma(U\backslash\{0\})\subset X\backslash Z$. Since $\pi$ is biholomorphic over $X\backslash Z$, $\gamma|_{U\backslash\{0\}}$ can be lifted to $\gamma':U\backslash\{0\}\to X'$ such that $0\notin\overline{\gamma'^{-1}(\pi^{-1}(Z))}$. Then it admits a holomorphic extension $\overline{\gamma'}:U\to X'$. Now $\pi\circ\overline{\gamma'}:U\to X$ is a holomorphic extension of $\gamma|_{U\backslash\{0\}}$. This proves that $(X,Z)$ is a Picard pair.
	\end{proof}
	\begin{prop}\label{prop_BPT_BH}
		Let $(X,Z)$ be a Picard pair where $X$ is a compact K\"ahler space and $Z\subset X$ is a closed analytic subset containing $X_{\rm sing}$. Then $X\backslash Z$ is Borel hyperbolic.
	\end{prop}
	\begin{proof}
		By Lemma \ref{lem_Picardpair_modification}, one can take a desingularization of the pair $(X,Z)$ and assume that $X$ is a compact K\"ahler manifold and $Z$ is a simple normal crossing divisor. For the remainder of the proof, see \cite[Corollary C]{DLSZ}.
	\end{proof}
	\subsubsection{Finsler metrics associated with the analytic prolongations of Viehweg-Zuo Higgs sheaves}\label{section_VZ_Finsler_metric}
	We follow the ideas in \cite{DLSZ} and make a little refinement by using non-canonical prolongations. 
	
	Let $f:(X^o,B^o),A^o\to S^o$ be a weakly admissible family of $(d,\Phi_c,\Gamma,\sigma)$-lc stable minimal models over a quasi-projective smooth variety $S^o$ which determines a generically finite morphism $\xi^o:S^o\to M_{\rm lc}(d,\Phi_c,\Gamma,\sigma)$. Let $S$ be a smooth projective compactification of $S^o$ so that $D:=S\backslash S^o$ is a simple normal crossing divisor and $\xi^o$ can be extended to a morphism $\xi:S\to M_{\rm lc}(d,\Phi_c,\Gamma,\sigma)$.
	Let $\sL$ be an ample line bundle on $S$. Let $$\bigoplus_{p=0}^w L^p\subset \pi_{\ast}\left({_{<\pi^{-1}(D)+E}}H^{w}_{h^o}\right):=\bigoplus_{p=0}^w\pi_{\ast}\left({_{<\pi^{-1}(D)+E}}H^{w-p,p}_{h^o}\right)$$ be the Higgs subsheaf generated by $L^0=\sL\otimes\sO_S(D)\otimes I_Z$ as in Theorem \ref{thm_big_Higgs_sheaf}, where $I_Z$ is some coherent ideal sheaf on $S$ whose co-support $Z$ lies in $D$ and ${\rm codim}_S(Z)\geq2$. Then $$\sL\otimes I_Z\subset\pi_{\ast}\left({_{<E}}H^{w,0}_{h^o}\right)$$ and $\bigoplus_{p=0}^w L^p\otimes\sO_S(-D)$ is a log Higgs subsheaf of $\pi_{\ast}\left({_{<E}}H^{w}_{h^o}\right)$ such that
	\begin{align*}
		\theta(L^p\otimes\sO_S(-D))\subset L^{p+1}\otimes\sO_S(-D)\otimes\Omega_S(\log D),\quad\forall p=0,\dots,w-1.
	\end{align*}	
	Consider the diagram
	\begin{align*}
		L^0\otimes\sO_S(-D)\stackrel{\theta}{\to}L^1\otimes\sO_S(-D)\otimes\Omega_S(\log D)\stackrel{\theta\otimes{\rm Id}}{\to}L^2\otimes\sO_S(-D)\otimes\Omega^{\otimes 2}_S(\log D)\to\cdots.
	\end{align*}
	This induces a map
	$$\tau_p:\sL\otimes I_Z \subset\pi_{\ast}\left({_{<E}}H^{w,0}_{h^o}\right)\to\pi_{\ast}\left({_{<E}}H^{w-p,p}_{h^o}\right)\otimes\Omega^{\otimes p}_S(\log D)$$ for every $p=0,\dots,\dim S$. 
	This induces a map
	\begin{align}\label{align_first_theta}
		\rho_p:T^{\otimes p}_S(-\log D)\to \sL^{-1}\otimes \pi_{\ast}\left({_{<E}}H^{w-p,p}_{h^o}\right).
	\end{align}
	The following lemma is essentially due to Deng \cite{Deng2022}, while we make a mild modification in order to adapt the result to the context of analytic prolongations.
	\begin{lem}\label{lem_rho1_generic_inj}
		$\rho_1:T_S(-\log D)\to \sL^{-1}\otimes \pi_{\ast}\left({_{<E}}H^{w-1,1}_{h^o}\right)$ is generically injective.
	\end{lem}
	\begin{proof}
		See \cite{Deng2022}. Notice that the metric $h_Q$ on $\pi_{\ast}\left({_{<E}}H^{w-1,1}_{h^o}\right)$ is bounded near the boundary $D$.
	\end{proof}
	Let $U_0\subset S\backslash D$ be a dense Zariski open subset such that $\rho_1$ is injective on $U_0$. Let $\gamma:\Delta^\ast\to S\backslash D$ be a holomorphic map such that ${\rm Im}(\gamma)\cap U_0\neq\emptyset$.
	By Lemma (\ref{lem_rho1_generic_inj}), the natural map
	$$\tau_{\gamma,1}:\gamma^\ast (\sL\otimes I_Z)\stackrel{\gamma^\ast\tau_1}{\to} \gamma^\ast\pi_{\ast}\left({_{<E}}H^{w-1,1}_{h^o}\right)\otimes\gamma^\ast\Omega_S(\log D)\stackrel{{\rm Id}\otimes d\gamma}{\to}\gamma^\ast\pi_{\ast}\left({_{<E}}H^{w-1,1}_{h^o}\right)\otimes\Omega_{\Delta^\ast}$$
	is nonzero. Hence
	there is a minimal integer $1\leq n_\gamma\leq \dim S$ such that 
	$$\tau_{\gamma,n_\gamma}:\gamma^\ast (\sL\otimes I_Z)\stackrel{\gamma^\ast\tau_{n_\gamma}}{\to} \gamma^\ast\pi_{\ast}\left({_{<E}}H^{w-n_\gamma,n_\gamma}_{h^o}\right)\otimes\gamma^\ast\Omega^{\otimes n_\gamma}_S(\log D)\stackrel{{\rm Id}\otimes d\gamma}{\to}\gamma^\ast\pi_{\ast}\left({_{<E}}H^{w-n_\gamma,n_\gamma}_{h^o}\right)\otimes\Omega^{\otimes n_\gamma}_{\Delta^\ast}$$
	is non-zero, whereas the composition
	$$\gamma^\ast (\sL\otimes I_Z) \stackrel{\tau_{\gamma,n_\gamma}}{\to}\gamma^\ast\pi_{\ast}\left({_{<E}}H^{w-n_\gamma,n_\gamma}_{h^o}\right)\otimes\Omega^{\otimes n_\gamma}_{\Delta^\ast}\to\gamma^\ast\pi_{\ast}\left({_{<E}}H^{w-n_\gamma-1,n_\gamma+1}_{h^o}\right)\otimes\Omega^{\otimes (n_\gamma+1)}_{\Delta^\ast}$$
	is zero. Then one has a non-zero map
	$$\sL\otimes I_Z\to \pi_{\ast}\left({_{<E}}H^{w-n_\gamma,n_\gamma}_{h^o}\right)\otimes S^{n_\gamma}\Omega_S(\log D).$$
	This induces a map 
	\begin{align}
		\iota_{n_\gamma}:  T^{\otimes n_\gamma}_S(-\log D)\to \sL^{-1}\otimes \pi_{\ast}\left({_{<E}}H^{w-n_\gamma,n_\gamma}_{h^o}\right).
	\end{align}
	Let $h_\sL$ be a hermitian metric on $\sL$ with positive curvature form and let $h_Q$ be the Hodge metric which is regarded as a singular hermitian metric on $\pi_{\ast}\left({_{<E}}H^{w-n_\gamma,n_\gamma}_{h^o}\right)$. Then the pullback $h_\gamma:=\iota_{n_\gamma}^\ast(h_\sL^{-1}h_Q)$ induces a Finsler pseudometric on $T_S(-\log D)$ by
	$$|v|_{h_\gamma}:=|\iota_{n_\gamma}(v^{\otimes n_\gamma})|^{\frac{1}{n_\gamma}}_{h_\sL^{-1}h_Q},\quad v\in T_S(-\log D).$$ The following proposition is essentially due to Deng-Lu-Sun-Zuo \cite{DLSZ}.
	\begin{prop}
		Notations as above. Then $|\frac{\partial}{\partial z}|^2_{\gamma^\ast h_\gamma}$ is not identically zero and the following inequality holds in the sense of currents
		$$\partial\dbar\log\left|\frac{\partial}{\partial z}\right|^2_{\gamma^\ast h_\gamma}\geq\frac{1}{n}\gamma^\ast(\Theta_{h_\sL}(\sL)).$$
	\end{prop}
	\begin{proof}
		Notice that $\gamma^\ast h_\gamma$ is a (possibly degenerate) smooth hermitian metric on $T_{\Delta^\ast}$. The first claim follows from the fact that $\tau_{\gamma,n_\gamma}$ is non-zero.
		By the Poincar\'e-Lelong equation, one has
		$$\partial\dbar\log\left|\frac{\partial}{\partial z}\right|^2_{\gamma^\ast h_\gamma}=-\Theta_{\gamma^\ast h_\gamma}(T_{\Delta^\ast})+R$$
		where $R$ is the ramification divisor of $\gamma$. Let  $N$ denote the saturation of the image of $d\gamma: T_{\Delta^\ast}\to\gamma^\ast T_S(-\log D)$. Then one knows that 
		\begin{align*}
			\Theta_{\gamma^\ast h_\gamma}(T_{\Delta^\ast})&\leq \Theta_{\gamma^\ast h_\gamma}(N)=\frac{1}{n_{\gamma}}\Theta_{\gamma^\ast h^{n_{\gamma}}_\gamma}(N^{\otimes n_{\gamma}})\\\nonumber
			&\leq \frac{1}{n_{\gamma}}\gamma^\ast\Theta_{ h^{n_{\gamma}}_\gamma}\left(T^{\otimes n_{\gamma}}_S(-\log D)\right)|_{N^{\otimes n_\gamma}}\\\nonumber
			&\leq \frac{1}{n_{\gamma}}\gamma^\ast\Theta_{h_\sL^{-1}h_Q}\left(\sL^{-1}\otimes \pi_{\ast}\left({_{<E}}H^{w-n_\gamma,n_\gamma}_{h^o}\right)\right)\big|_{\gamma^\ast(\iota_{n_\gamma}(N^{\otimes n_\gamma}))}\\\nonumber
			&=-\frac{1}{n_{\gamma}}\gamma^\ast\Theta_{h_\sL}(\sL)+\frac{1}{n_{\gamma}}\gamma^\ast\Theta_{h_Q}\left( \pi_{\ast}\left({_{<E}}H^{w-n_\gamma,n_\gamma}_{h^o}\right)\right)\big|_{\gamma^\ast(\iota_{n_\gamma}(N^{\otimes n_\gamma}))}
		\end{align*}
		as $(1,1)$-forms on $\Delta^\ast$.
		By the definition of $n_\gamma$, one sees that $\gamma^\ast(\iota_{n_\gamma}(N^{\otimes n_\gamma}))$ lies in the kernel of the Higgs field 
		$$\theta_\gamma:\gamma^\ast\pi_{\ast}\left({_{<E}}H^{w-n_\gamma,n_\gamma}_{h^o}\right)\to \gamma^\ast\pi_{\ast}\left({_{<E}}H^{w-n_\gamma-1,n_\gamma+1}_{h^o}\right)\otimes\Omega_{\Delta^\ast}$$
		of the Higgs sheaf $\gamma^\ast\pi_{\ast}\left({_{<E}}H^{w}_{h^o}\right)$.
		By Griffiths' curvature formula
		$$\gamma^\ast\Theta_{h_Q}(H^w_{h^o})+\theta_{\gamma}\wedge\overline{\theta_{\gamma}}+\overline{\theta_{\gamma}}\wedge\theta_{\gamma}=0,$$
		one gets that
		$$\gamma^\ast\Theta_{h_Q}\left( \pi_{\ast}\left({_{<E}}H^{w-n_\gamma,n_\gamma}_{h^o}\right)\right)\big|_{\gamma^\ast(\iota_{n_\gamma}(N^{\otimes n_\gamma}))}=-\theta_{\gamma}\wedge\overline{\theta_{\gamma}}|_{\gamma^\ast(\iota_{n_\gamma}(N^{\otimes n_\gamma}))}\leq0.$$ 
		Combining the formulas above, the proposition is proved.
	\end{proof}
	\subsubsection{Admissible families and big Picard theorem}
	\begin{thm}\label{thm_main_proof}
		Let $f^o:(X^o,B^o),A^o\to S^o$ be an admissible family of $(d,\Phi_c,\Gamma,\sigma)$-lc stable minimal models over a smooth quasi-projective variety $S^o$ which defines a quasi-finite morphism $\xi^o:S^o\to M_{\rm lc}(d,\Phi_c,\Gamma,\sigma)$.
		Let $S$ be a projective variety containing $S^o$ as a Zariski open subset. Then $(S,S\backslash S^o)$ is a Picard pair. Consquently, $S^o$ is Borel hyperbolic.
	\end{thm}
	\begin{proof}
		Let $\gamma:\Delta^\ast\to S$ be a holomorphic map such that $0\notin\overline{\gamma^{-1}(S\backslash S^o)}$. We are going to show that $\gamma$ extends holomorphically to $0\in\Delta$. By shrinking $\Delta$ we may assume that ${\rm Im}(\gamma)\subset S^o$.
		Take $B\subset S$ to be the Zariski closure of ${\rm Im}(\gamma)$. Let $\pi:B'\to B$ be a desingularization so that $\pi$ is biholomorphic over $B_{\rm reg}\cap S^o$ and $\pi^{-1}(B\backslash(B_{\rm reg}\cap S^o))$ is a simple normal crossing divisor on $B'$. Since $\pi$ is a proper map, $\gamma$ can be lifted to $\gamma':\Delta^\ast\to \pi^{-1}(B\cap S^o)$. It suffices to show that $\gamma'$ extends to a holomorphic map $\overline{\gamma'}:\Delta\to B'$. Taking the base change of $f^o$ via $\pi^{-1}(B\cap S^o)\to S^o$, we may assume the following without loss of generality.
		\begin{framed}
			$S$ is smooth, $D:=S\backslash S^o$ is a simple normal crossing divisor on $S$ and  $\gamma:\Delta^\ast\to S\backslash D$ is a holomorphic map such that ${\rm Im}(\gamma)$ is Zariski dense in $S$.
		\end{framed}
		We are going to show that $\gamma$ extends to a holomorphic map $\overline{\gamma}:\Delta\to S$.   
		By \S \ref{section_VZ_Finsler_metric} there is a Finsler pseudometric $h_\gamma$ on $T_S(-\log D)$ such that $|\frac{\partial}{\partial z}|^2_{\gamma^\ast h_\gamma}$ is not identically zero and that the following inequality holds in the sense of currents
		$$\partial\dbar\log\left|\frac{\partial}{\partial z}\right|^2_{\gamma^\ast h_\gamma}\geq\gamma^\ast(\omega_S),$$
		where $\omega_S$ is a hermitian metric on $S$. Thus $\gamma$ extends to a holomorphic map $\overline{\gamma}:\Delta\to S$ by Theorem \ref{thm_big_Picard_LSZ}. This proves that $(S,S\backslash S^o)$ is a Picard pair. The claim that $S^o$ is Borel hyperbolic follows from Proposition \ref{prop_BPT_BH}.
	\end{proof}
	
	\subsection{Pseudo Kobayashi hyperbolicity}
	Theorem \ref{thm_big_Higgs_sheaf}, combined with the method in \cite{Deng2022}, implies the pseudo Kobayashi hyperbolicity of the base of a weakly admissible family of lc stable minimal models. For the basic notions of pseudo Kobayashi hyperbolicity, the reader is referred to \cite{Deng2022}.
	\begin{thm}
		Let $f^o:(X^o,B^o), A^o\to S^o$ be a weakly admissible family of $(d,\Phi_c,\Gamma,\sigma)$-lc stable minimal models over a smooth quasi-projective variety $S^o$ which defines a generically finite morphism $\xi^o:S^o\to M_{\rm lc}(d,\Phi_c,\Gamma,\sigma)$. Then $S^o$ is pseudo Kobayashi hyperbolic.
	\end{thm}
	\begin{proof}
		The arguments in \cite{Deng2022} is applicable to our case. According to \cite[\S 2.3]{Deng2022} and (\ref{align_first_theta}), one obtains a Finsler pseudometric
		$$F:=\left(\sum_{k=1}^wk\alpha_k\rho_k^\ast(h_{\sL}^{-1}h_Q)^{\frac{2}{k}}\right)^{\frac{1}{2}},$$
		which is positive definite on some Zariski open subset of $S^o$ and is bounded from above by a negative constant for some $\alpha_1,\dots,\alpha_w>0$. Consequently, $S^o$ is pseudo Kobayashi hyperbolic by \cite[Lemma 2.4]{Deng2022}.	
	\end{proof}
	\subsection{Brody hyperbolicity}
	\begin{thm}\label{thm_main_Brody_proof}
		Let $f^o:(X^o,B^o), A^o\to S^o$ be an admissible family of $(d,\Phi_c,\Gamma,\sigma)$-lc stable minimal models over a smooth quasi-projective variety $S^o$ which defines a quasi-finite morphism $\xi^o:S^o\to M_{\rm lc}(d,\Phi_c,\Gamma,\sigma)$.
		Then $S^o$ is Brody hyperbolic.
	\end{thm}
	\begin{proof}
		Assume that there is a holomorphic map $h:\bC\to S^o$. By Theorem \ref{thm_main_proof}, $h$ can be extended to an algebraic morphism $\overline{h}:\bP^1\to S$. Taking the base change of the family $f^o$ via the map $h$, we obtain an admissible (algebraic) family $f^o|_{\bC}$ of $(d,\Phi_c,\Gamma,\sigma)$-lc stable minimal models over $\bC$. Since $\bC$ is not of log general type, Theorem \ref{thm_VH} implies that the classifying map $\xi^o\circ h$ of $f^o|_{\bC}$ must be constant. Consequently, $h$ is constant due to the quasi-finiteness of $\xi^o$. This proves the theorem.
	\end{proof}
	\subsection{Stratified hyperbolicity}
	Based on the results in the previous sections, we introduce the following notion.
	\begin{defn}
		Let $f:(X,B),A\to S$ be a family of lc stable minimal models over a quasi-projective variety $S$ such that the coefficients of $B$ lie in $[0,1)$. An \emph{admissible stratification} of $S$ with respect to $f$ is a filtration of Zariski closed subsets
		$$\emptyset=S_{-1}\subset S_0\subset\cdots\subset S_d=S$$
		satisfying the following conditions:
		\begin{itemize}
			\item $d=\dim S$,
			\item  $S_i\backslash S_{i-1}$ is a (possibly disconnected) smooth variety  of dimension $i$ for each $i=0,\dots,d$,
			\item each $\overline{S_i}\backslash S_i$ is a disjoint union of strata of $\{S_i\}$, and
			\item the pullback family of $f$ over $S_i\backslash S_{i-1}$ is admissible for each $i=0,\dots,d$.
		\end{itemize}
	\end{defn}
	The following lemma shows that an admissible stratification always exists.
	\begin{lem}
		Let $f:(X,B),A\to S$ be a family of lc stable minimal models over a quasi-projective variety $S$ such that the coefficients of $B$ lie in $[0,1)$. Then there is a dense Zariski open subset $U\subset S_{\rm reg}$ such that $f|_U$ is admissible.
	\end{lem}
	\begin{proof}
		Take a log resolution $\pi:X'\to X$ of the pair $(X,\operatorname{supp}(A+B))$. By generic smoothness, there is a dense Zariski open subset $U\subset S_{\rm reg}$ such that $(X',\pi^{\ast}(A+B))$ is a log smooth family over $U$ and $\pi|_{X'_s}:X'_s\rightarrow X_s$ is a birational morphism for every $s\in U(\bC)$. This proves the lemma.
	\end{proof}
	The following theorem is a direct consequence of Theorem \ref{thm_main}.
	\begin{thm}\label{thm_stratified_hyperbolic}
		Let $f:(X,B),A\to S$ be a family of $(d,\Phi_c,\Gamma,\sigma)$-lc stable minimal models over a quasi-projective variety $S$ such that the coefficients of $B$ lie in $[0,1)$. Assume that $f$ induces a quasi-finite morphism $S\to M_{\rm lc}(d,\Phi_c,\Gamma,\sigma)$. Let $$\emptyset=S_{-1}\subset S_0\subset\cdots\subset S_{\dim S}=S$$ be an admissible stratification of $S$ with respect to $f$. Then the following hold for each $i=0,\dots,\dim S$. 
		\begin{itemize}
			\item\emph{(big Picard theorem)} $(S_i,S_{i-1})$ is a Picard pair.
			\item\emph{(Borel hyperbolicity)} Any holomorphic map from an algebraic variety to $S_i\backslash S_{i-1}$ is algebraic. 
			\item\emph{(Viehweg hyperbolicity)} $S_i\backslash S_{i-1}$ is of log general type.
			\item\emph{(pseudo Kobayashi hyperbolicity)} $S_i\backslash S_{i-1}$ is Kobayashi hyperbolic away from a proper Zariski closed subset.
			\item\emph{(Brody hyperbolicity)} $S_i\backslash S_{i-1}$ is Brody hyperbolic.
		\end{itemize}
	\end{thm}
	\bibliographystyle{plain}
	\bibliography{Hyper_Moduli}
\end{document}